\documentclass[10pt,A4paper,fleqn]{amsart}
\usepackage{mathrsfs}
\usepackage{amsfonts}
\usepackage{txfonts}
\usepackage{amsmath}
\usepackage{amssymb}
\usepackage{amsthm}
\usepackage[pdftex]{graphicx}
\usepackage[toc,page,title,titletoc,header]{appendix}
\usepackage{geometry}
\usepackage{upgreek}
\usepackage[T1]{fontenc}
\usepackage{color}
\usepackage{hyperref}
\usepackage[all]{xy}
\usepackage[french,english]{babel}

\theoremstyle{plain}
\newtheorem{theo}{Theorem}[section]
\newtheorem*{theo*}{Theorem}
\newtheorem{prop}{Proposition}[section]

\newtheorem{cor}{Corollary}[section]
\newtheorem*{lem*}{Lemma}

\theoremstyle{definition}
\newtheorem{lem}{Lemma}[section]

\newtheorem{question}{Question}[section]

\theoremstyle{remark}

\newtheorem*{rem*}{Remark}
\newtheorem{rem}{Remark}[section]

\newtheorem{claim}{Claim}[section]

\newcommand{\R}{\mathbb{R}}
\newcommand{\T}{\mathbb{T}}
\newcommand{\Z}{\mathbb{Z}}
\newcommand{\C}{\mathbb{C}}

\begin{document}

\bibliographystyle{alpha}

\abstract 

In this article, we investigate the mathematical part of De Sitter's theory on the Galilean satellites, and further extend this theory by showing the existence of some quasi-periodic librating orbits by applications of KAM theorems. After showing the existence of De Sitter's family of linearly stable periodic orbits in the Jupiter-Io-Europa-Ganymede model by averaging and reduction techniques in the Hamiltonian framework, we further discuss the possible extension of this theory to include a fourth satellite Callisto, and establish the existence of a set of positive measure of quasi-periodic librating orbits in both models for almost all choices of masses. 

\endabstract

\title{De Sitter's Theory of Galilean Satellites and the related quasi-periodic orbits}

\author{Henk Broer, Lei Zhao}


\date\today
\maketitle

\tableofcontents

\section{Introduction}
Based on the methods of studying celestial motions of planets by pure Keplerian approximations with additional corrections (called \emph{inequalities}), formalizations of a mathematical theory of the Jovian Galilean satellites has been facing great difficulty for a long time. Indeed, even though this five-body system seems to be very much like a rescaled model of the inner solar system, and the mean motion, \emph{i.e.} the Keplerian elliptic orbital frequency of Callisto is not close to lower-order resonance with the three inner satellites, the mean motions of Io, Europa and Ganymede are close to a $1:2:4$ resonance, which give rise to inequalities greatly effecting the real behaviors of the system. 

In terms of averaging theory, we understand that the averaged system along this resonance differs by the corresponding (first-order) secular system, which is the averaged system of the perturbation over the three fast Keplerian motions, by some \emph{critical terms} (\emph{i.e.} resonant terms), corresponding to the $1:2$ mean motion resonances of Io-Europa and Europa-Ganymede. The order of these resonance is so low such that these critical terms, in contrast to the critical term of high-order resonances, \emph{e.g.} the $2:5$ resonance roughly satisfied by the mean motions of Jupiter and Saturn, dominates, instead of being dominated by, the secular system. 
Moreover, not being first integrals of a first order approximating system (obtaining by adding the corresponding critical terms to the Keplerian part), the variations of the semi major axes, and therefore the contribution of the Keplerian part to the dynamics of the semi-fast variables near the resonances could also be significant.

It is therefore important to analyze the dynamics of the first order approximating system obtained by adding up the corresponding critical terms, and to investigate whether it can be qualitatively continued when higher order effects are taken into account. These, together with a sufficient approximation of the normalizing transformation involved in the averaging procedure, provide a good theory for their celestial motions. This summarizes the strategy of W. de Sitter's theory \cite{DeSitter1925}, \cite{DeSitter1931} of the Jovian Galilean satellites: Io, Europa, Ganymede and Callisto. 

As a starting point, De Sitter neglected Callisto and only considered a planar 4-body problem modeling the Jupiter-Io-Europa-Ganymede system. In modern terms, we may interpret this as follows: After being averaged over the $1:2:4$-resonance, truncated at the first orders of the eccentricities and of the masses, and further reduced by the rotational$SO(2)$-symmetry of this system, he established the existence of several families of periodic solutions with non-circular but almost-circular Keplerian ellipses, parametrized by one of the eccentricities of the satellites. Along these orbits, the semi major axes and the eccentricities of the satellites do not change. In a fixed reference frame, the pericentres (or ``perijoves'' when the massive body is called ``Jupiter'') of the satellites process uniformly. De Sitter found this to be an important advantage of his theory (\cite[p. 8]{DeSitter1925}). Indeed, in comparison with a theory built only on pure Keplerian motions, his first-order approximating system successfully removed the dynamical degeneracy of the uncoupled Keplerian systems that all their bounded orbits are closed, and very much dominates the local dynamics of the higher-order approximating systems and the full problem. Moreover, by a continuation method of Poincaré, all these families of periodic orbits admit continuations to higher-order approximating systems and to the full system. Only one of these families of periodic solutions is found to be linear stable, to which, with properly assigned parameters, models and interprets the real motions of Io, Europa, and Ganymede in our solar system.

\begin{theo}(De Sitter, \cite{DeSitter1909})\label{Theo: 1.1} In the planetary $4$-body problem with one mass sufficiently dominates the others, after symplectically reduced by the translation and rotational symmetry, there exists a family of linearly stable periodic orbits with $4:2:1$-mean motion resonance. 
\end{theo}

In this article, we shall first represent these periodic orbits of the planar $4$-body problem with some modern explanations and clarifications of this work \cite{DeSitter1909} of De Sitter. For comparison purpose, we shall also treat what we would call $1+3$-body problem of one fixed center and three satellites in parallel and establish the same result. Following De Sitter, we shall, after being reduced by the $SO(2)$-symmetry of rotations, first find particular periodic solutions of an approximating system, verify that they can be continued to periodic orbits of the full system, and identify the only linearly stable periodic solutions.


In addition to this, for almost all masses, we shall establish the existence of librating quasi-periodic KAM orbits around the linearly stable family of periodic orbits. Since the real motions of Io, Europa, and Ganymede only \emph{roughly} satisfy the $1:2:4$-resonance, these quasi-periodic orbits could serve to complete De Sitter's theory of these satellites.  

\begin{figure}\label{Fig: collinearity}
\centering
\includegraphics[width=3in]{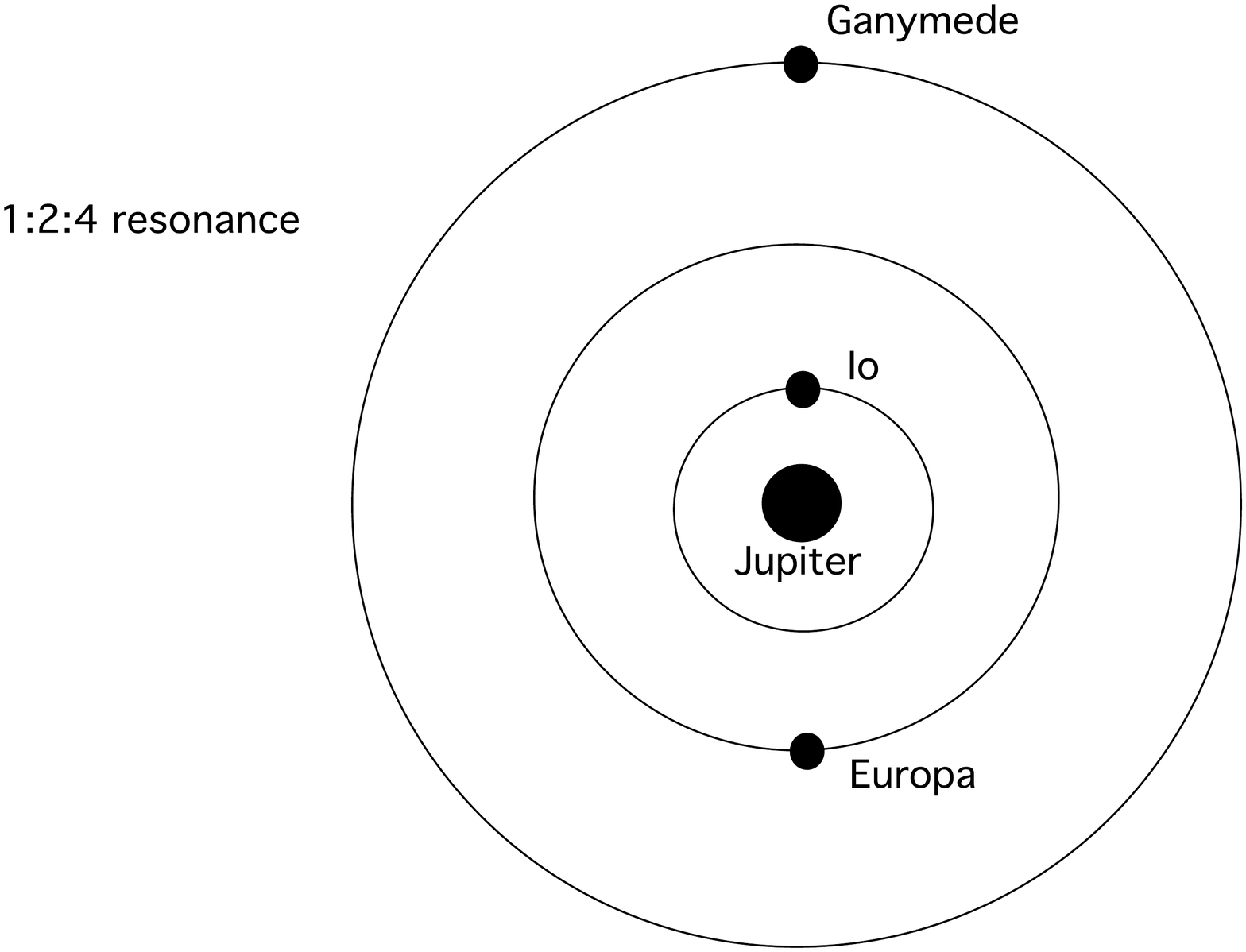}
\caption{Jupiter-Io-Europa-Ganymede system at a collinearity}
\end{figure}

\begin{theo} In both $1+3$-body problem or $4$-body problem symplectically reduced by the translations and rotations,  for almost all masses among which one sufficiently dominates the others, there exists a set of positive measure of quasi-periodic orbits librating around the family of linearly stable periodic orbits as indicated in Thm \ref{Theo: 1.1}.
\end{theo}

In the analysis above, we have completely ignored Callisto, which is more distant from the Jupiter compared to the inner three satellites, but is nevertheless important to be included in the theory, since it might cause important secular inequalities of the inner three satellites. To have a more complete theory of the Galilean satellites including Callisto, we proceed by adding an outer fourth satellite to the system.  Near circular motions of Callisto, De Sitter's theory can be directly extended to such a model, in which we establish the existence of several families of quasi-periodic orbits, with some of them correspond to De Sitter's periodic orbits, some of them correspond to librating quasi-periodic solutions established above. 

\begin{theo} In both $1+4$-body problem or $5$-body problem symplectically reduced by the translations and rotations, for almost all masses among which one sufficiently dominates the others, there exist
\begin{itemize}
\item a set of quasi-periodic orbits lying on a family of normally elliptic invariant 2-tori, along which the three inner mean motions satisfy the $4:2:1$ resonance, and the outermost orbit is almost circular with frequency incommensurable with the frequencies of the inner three; and
\item a set of positive measure of quasi-periodic orbits, in which the three inner mean motions are close to $4:2:1$ resonance, and the outermost orbit is almost circular with frequency incommensurable with the frequencies of the inner three.
\end{itemize}
\end{theo}

We organize this article as follows: In Section \ref{Section: 1+3 and 4 body problem} we formulate our models and present them in canonical Joviancentric coordinates. A procedure of calculating the critical part of the perturbing function and the associated averaging procedure are presented in Section \ref{Section: Calculation perturbing function}. Section \ref{Section: Periodic Orbits of De Sitter} presents De Sitter's analysis on his family of linearly stable periodic orbits. We then present a version of KAM theorem in Section \ref{Section: KAM theorem}, which will be applied in Section \ref{Sec: KAM near periodic orbit} to establish librating quasi-periodic orbits around these periodic orbits, and in Section \ref{Section: A complete system of Galilean satellites} to extend these results to the 1+4 / 5-body problems by establishing the existence of families of quasi-periodic orbits.

We aim to present the main ideas with many more explanations in another expository work \cite{BZ2}, which we hope could serve a broader community.


\section{Formulations: The $1+3$-Body Problem and the $4$-Body Problem}\label{Section: 1+3 and 4 body problem}
\subsection{The $1+3$-Body Problem}
To have a simplest model of our problem, we shall start by considering the planar $1+3$-body problem with a fixed center with mass $m_{0}$ and three small bodies with masses $m_{1}, m_{2}, m_{3} << m_{0}$. In view of the fact that the mutual interactions of the bodies are smaller than the attracting forces from the center body, we may decompose the Hamiltonian $F$ of the system as
$$F=F_{Kep}+F_{pert},$$
with 
$$F_{Kep}=\dfrac{\|p_{1}\|^{2}}{m_{1}}+\dfrac{\|p_{2}\|^{2}}{m_{2}}+\dfrac{\|p_{3}\|^{2}}{m_{3}}-\dfrac{m_{0} m_{1}}{r_{01}}-\dfrac{m_{0} m_{2}}{r_{02}}-\dfrac{m_{0} m_{3}}{r_{03}}$$
$$F_{pert}=-\dfrac{m_{1} m_{2}}{r_{12}}-\dfrac{m_{2} m_{3}}{r_{23}}-\dfrac{m_{3} m_{1}}{r_{31}}$$
in which\footnote{Readers comparing this article with \cite{DeSitter1909} should be careful that De Sitter took $R=-F_{pert}$.} we have denoted
the positions and momenta as 
$$\qquad (q_{1}, q_{2}, q_{3}, p_{1}, p_{2}, p_{3}) \in T^{*} (\{(q_{1}, q_{2}, q_{3}) \in \R^{2} \times \R^{2} \times \R^{2} : q_{1} \neq q_{2}, q_{2} \neq q_{3}, q_{3} \neq q_{1}\})\,$$
 $$\hbox{ and the mutual distances } r_{ij}=\|q_{i}-q_{j}\|, \, 0 \le i < j \le 3, \hbox{ in which we have set }q_{0} \equiv 0.$$

\subsection{The $4$-Body Problem in Canonical Joviancentric Coordinates}
In parallel, we shall also consider the planar full 4-body problem with masses $m_{1}, m_{2}, m_{3} << m_{0}$, positions $q_{0}, q_{1}, q_{2}, q_{3}$ and momenta $p_{0}, p_{1}, p_{2}, p_{3}$ and write, following \cite{LaskarRobutel},  the Hamiltonian in the canonical Joviancentric coordinates 
$$(\tilde{p}_{0}, \tilde{p}_{1}, \tilde{p}_{2}, \tilde{p}_{3}, \tilde{q}_{0}, \tilde{q}_{1}, \tilde{q}_{2}, \tilde{q}_{3}),$$ defined by
\begin{equation*} 
\left\{
\begin{array}{ll}\tilde{q}_{0}=q_{0}   \\ \tilde{q}_{1}=q_{1}-q_{0} \\ \tilde{q}_{2}=q_{2}-q_{0} \\ \tilde{q}_{3}=q_{3}-q_{0}.
\end{array}\right.
\phantom{aaaaaaa}
\left\{
\begin{array}{ll}\tilde{p}_{0}=p_{0}+p_{1}+p_{2}+p_{3}   \\ \tilde{p}_{1}=p_{1}\\ \tilde{p}_{2}=p_{2} \\ \tilde{p}_{3}=p_{3}.
\end{array}\right.
\end{equation*}

The kinetic part thus becomes
$$T=T_{0}+T_{1}$$
with 
$$T_{0}=\frac{1}{2} \sum_{i=1}^{n} \|\tilde{p}_{i}\|^{2} (\frac{1}{m_{i}}+\frac{1}{m_{0}}),\quad T_{1}=\sum_{0<i<j} \frac{\tilde{p}_{i} \cdot \tilde{p}_{j}}{m_{0}}; $$
The potential part simply takes the form
$$U=U_{0}+U_{1}$$
with
$$U_{0}=-\sum_{i=1}^{n} \frac{m_{0} m_{i}}{r_{0 i}}, \quad U_{1}=-\sum_{0<i<j} \frac{m_{i} m_{j}}{r_{i j}}.$$
The complete Hamiltonian can thus be decomposed as the Keplerian part $F_{Kep}=T_{0}+U_{0}$ and the perturbing part $F_{pert}=T_{1}+U_{1}.$ Compare with the $1+3$-body formulation, the fact that Jupiter is not fixed causes the modification of the mass parameters in the Keplerian part, and the appearance of the indirect part $T_{1}$, which, though well-known to have no secular influence, does affect the study of lower order resonances (the reader is invited to consult the end of Section \ref{Section: Calculation perturbing function} for this point). 

In this article, we shall treat both models altogether while commenting on their differences.

\subsection{Introduction of a Small Parameter: Order of the Mass Ratio}
Let $\mu$ be a small parameter representing the order of the mass ratio between $m_{i}, i=1,2,3$ and $m_{0}$. The momenta $\tilde{p}_{i}$ is of order $\mu$ if the velocities are considered as of order $1$. We thus set 
$$m_{1}=\mu \bar{m}_{1}, m_{2}=\mu \bar{m}_{2}, m_{3}=\mu \bar{m}_{3}, \tilde{p}_{1}=\mu \bar{p}_{1}, \tilde{p}_{2}=\mu \bar{p}_{2}, \tilde{p}_{3}=\mu \bar{p}_{3}.$$
Therefore, we have $F_{Kep} \sim \mu$ and $F_{pert}=O(\mu^{2})$.

We rescale the symplectic form by taking $(\bar{p}_{1},\bar{p}_{2}, \bar{p}_{3}, \tilde{q}_{1}, \tilde{q}_{2}, \tilde{q}_{3})$ instead of $(\tilde{p}_{1},\tilde{p}_{2}, \tilde{p}_{3}, \tilde{q}_{1}, \tilde{q}_{2}, \tilde{q}_{3})$ as a set of Darboux coordinates. To preserve the same Hamiltonian vector field, we have to rescale the Hamiltonian function by a factor of $\mu^{-1}$. After the rescaling, we have (by abuse of notations) $F_{Kep} \sim 1$ and $F_{pert}=O(\mu)$.

\subsection{Poincaré's Continuation Method} \label{Subsection: Poincaré's continuation method}
When $\mu=0$, the system $F_{Kep}$ determines three uncoupled Keplerian motions. This is a \emph{properly degenerate} system. This proper degeneracy arise from the fact that in the Kepler problem all bounded orbits are closed, and is, very often, both a bless (the unperturbed dynamics is simple) and a difficulty (the perturbed dynamics of a degenerate systems cannot be well-controlled and higher order effects have to be considered to, hopefully, remove the degeneracy).

In \cite{Poincare1892}, Poincaré considered the problem of continuation of periodic orbits from proper-degenerate systems.  

Let $(I, \theta)=(I_{1}, \theta_{1}, I_{2}, \theta_{2})$ be a set of symplectic coordinates with $(I_{1}, \theta_{1}) \in T^{*} \T^{n_{1}}$ and $(I_{2}, \theta_{2}) \in T^{*} \T^{n_{2}}$,
$$F(I, \theta)=F_{0}(I_{1})+\mu \cdot F_{1} ({I_{1}=I_{1}^{0}}, \theta'_{1},I_{2}, \theta_{2})+o(\mu)$$
be a Hamiltonian function depending on a small parameter $\mu$. For small $\mu$, the dependence of $\mu F_{1}$ on $I_{1}$ is much weaker than the dependence of $F_{0}$ on these variables, therefore provided some non-degeneracy on $F_{0}$ is satisfied, we could fix $I_{1}$ in the expression of $F_{1}$. Along a periodic solution $\xi$ of $F_{0}$, the frequencies $n_{1}=\dfrac{d F_{0}}{d I_{1}}$ are in complete resonance; we have denoted resonant angles (or \emph{critical arguments}) by $\theta'_{1}$. 

\begin{theo}(Poincaré,\cite[n.46, p. 133]{Poincare1892})\label{Theo: Poincare continuation} The periodic solution $\xi$ of $F_{0}$ can be continued to a periodic solution of $F$ for $0<\mu<<1$ provided that
\begin{itemize}
\item The Hessian of $F_{0}$ with respect to $I_{1}$ is non-degenerate: $\det\dfrac{d^{2} F_{0}}{d I_{1}^{2}} \neq 0$;
\item $\xi$ corresponds to a non-degenerate critical point of $F_{1}$.
\end{itemize}
\end{theo}

In the original statement of Poincaré, the second condition is weaker: It is enough to impose  that $\xi$ corresponds to a critical point of odd multiplicity of $F_{1}$. This ``non-degenerate'' version can eventually be phrased by the non-degeneracy of the associated Poincar\'e map. The proof follows from the implicit function theorem.

\subsection{First Order Approximating System and the Introduction of another Small Parameter: Order of the Eccentricities}
Let us consider an invariant resonant 6-tori of $F_{Kep}$ characterized by the $4:2:1$- resonant condition on the mean motions.

To apply Thm \ref{Theo: Poincare continuation}, we average $F_{pert}$ over this resonance in an $\hbox{Cst } \varepsilon$-neighborhood\footnote{We shall eventually only consider persistence of invariant objects under $O(\varepsilon)$-perturbations. The restriction is made to adapt to this.} $\tilde{N}$ (with a large enough constant Cst) of these resonant tori, thus results in a certain function $F_{res}$, containing only the \emph{critical terms}, or resonant terms corresponds to this resonance, and truncate at the first order of the eccentricities.  The function $F$ is thus conjugated to  
$$F_{Kep}+F_{res}+F_{rem},$$
in which $F_{Kep}$ is of order $1$, $F_{res}$ is of order $\mu e$, while the remainder $F_{rem}$ is of the order $O(\mu e^{2})+O(\mu^{2})$, in which $e$ is a small parameter characterizing the order of smallness of $e_{1}, e_{2}, e_{3}$. This averaging procedure will be made precise in Subsection \ref{Subsection: Elimination procedure}.

\section{Calculation of the Critical Part of the Perturbing Function}\label{Section: Calculation perturbing function}
Let $l_{1}, l_{2}, l_{3}$ and $g_{1}, g_{2}, g_{3}$ denote respectively the mean anomalies and the arguments of the pericentres of the particles $q_{1}, q_{2}, q_{3}$. We restricted ourselves to a deleted neighborhood of the circular motions to have these elements always well defined. These elements, together with the semi major axes $a_{1}, a_{2}, a_{3}$ and $e_{1}, e_{2}, e_{3}$ form a set of regular coordinates in such a deleted neighborhood in the phase space.  

To obtain the desired periodic solutions, we start by calculating the averaged perturbing function over the above-mentioned resonance, truncated at the first order of the eccentricities $e_{1}, e_{2}, e_{3}$ that we suppose to be small.

It is known (\cite[p.305]{Tisserand1889}) that the function $F_{pert}$ can be developed in Fourier series containing only cosines of the angles 
$$k_{1} l_{i}+k_{2} l_{j}+k_{3} g_{i}+ k_{4} g_{j}, (i,j)=(1,2), (2,3), (3,1), k_{1}, k_{2}, k_{3}, k_{4} \in \Z$$
with coefficients depending only on the semi major axes and the eccentricities. Among such terms we shall only be interested in those which are of first order in the eccentricities, and moreover, only those contain multiples of the critical arguments $\delta_{1}:=l_{1}-2 l_{2}$ and $\delta_{2}:=l_{2}-2 l_{3}$, since only these terms persist after being averaged over the resonance.

As indicated in \cite[p.307]{Tisserand1889}, to have terms at most of first order in the eccentricities, the corresponding $(k_{1}, k_{2}, k_{3}, k_{4})$ must satisfy
\begin{equation}\label{eq:1}
|k_{1}-k_{3}|+|k_{2}-k_{4}| \le 1.
\end{equation}
Moreover, due to the invariance of the potential function under the rotations, we must have$$k_{3}+k_{4}=0.$$
We see immediately that Eq. \ref{eq:1} cannot be satisfied for $k_{2}=4 k_{1} \neq 0$, \emph{i.e.} the mutual interaction of the innermost and the outermost bodies $q_{1}, q_{3}$ could appear only at the second order of the eccentricities and are completely negligible for our present calculation of $F_{res}$. For $k_{2}=2 k_{1}$, we can only have 
$$(k_{1}, k_{2}, k_{3}, k_{4})=(1,-2,1,-1)\hbox{ or }(1,-2,2,-2).$$
It is thus sufficient to calculate the coefficients of each of the two corresponding terms (which are actually of first order in the eccentricities) from the series expansions of $-\dfrac{m_{1} m_{2}}{r_{12}}$ and $-\dfrac{m_{2} m_{3}}{r_{23}}$ respectively. 

We now calculate, in the series expansions of $-\dfrac{m_{1} m_{2}}{r_{12}}$, the coefficients of the terms $\cos(l_{1}-2 l_{2}+g_{1}-g_{2})$ and $\cos(l_{1}-2 l_{2}+2 g_{1}-2 g_{2})$ while those in $-\dfrac{m_{1} m_{2}}{r_{12}}$ are completely analogous. Let $\alpha=\dfrac{a_{1}}{a_{2}}$ be the semi major axes ratio of the inner two bodies, $\rho=\dfrac{r_{1}}{r_{2}}$ their ratio of radii, and $v_{1}, v_{2}$ be the true anomalies of the inner two bodies. Following \cite{LaskarRobutel}, we let
\begin{equation} \label{eq: expansion 1}
\dfrac{1}{r_{12}}=\dfrac{1}{r_{2}} \cdot \dfrac{1}{\sqrt{A+V}},
\end{equation}
in which 
\begin{align*}&A=1+\alpha^{2}-2 \alpha \cos(l_{1}-l_{2}+g_{1}-g_{2}), \\ &V=2 \alpha \left(\cos(l_{1}-l_{2}+g_{1}-g_{2})-\dfrac{\rho}{\alpha} \cos \angle (q_{1}, q_{2})\right)+\rho^{2}-\alpha^{2}.
\end{align*} 

The function $V$ is at least of first order in the eccentricities $e_{1}, e_{2}$, hence is much smaller than $A$ in the deleted neighborhood we are now considering. We expand Eq. \ref{eq: expansion 1} into
\begin{equation} \label{eq: expansion 2}
\dfrac{1}{r_{12}}=\dfrac{1}{r_{2}} \cdot \dfrac{1}{\sqrt{A}}-\dfrac{1}{r_{2}}\dfrac{V}{\sqrt{A^{3}}}+O(V^{2}),
\end{equation}
The expansions of $A^{-s/2}, s=1,3$ in terms of $\cos k (l_{1}-l_{2}+g_{1}-g_{2}), k =0, 1, \cdots$ is obtained by the usual Laplace coefficients $b_{s}^{(k)}(\alpha), s=\frac{1}{2}, \frac{3}{2}$, defined as the coefficients of the Laurent series (c.f. \cite{LaskarRobutel})
$$A^{-s}=(1-\alpha z)^{-s} (1-\alpha z^{-1})^{-s}=\dfrac{1}{2} \sum^{+ \infty}_{k=- \infty} b_{s}^{(k)}(\alpha) \, z^{k}.$$

A calculation shows that in first order of the eccentricities, $r_{2}$ only contains terms with argument $l_{2}$, and that in first order of the eccentricities, $V$ only contains terms with arguments 
$$l_{1}, l_{2}-g_{1}+g_{2}, 2 l_{1} - l_{2} +g_{1}-g_{2}, l_{2}, l_{1}+g_{1}-g_{2}, l_{1}-2 l_{2}+g_{1}-g_{2}.$$
We observe that to calculate these terms, it is enough to only calculate, in the expansions of $A^{-1/2}$ and $A^{-3/2}$, those terms with arguments $0$, $l_{1}-l_{2}+g_{1}-g_{2}$, $2 l_{1} - 2 l_{2} + 2 g_{1} -2 g_{2}$ and $3 l_{1} - 3 l_{2} + 3 g_{1} -3 g_{2}$ (c.f. \cite{LaskarRobutel}).

The effective calculation is assisted by Maple 16. We find that 
\begin{align*}
F_{res}=&\dfrac{m_{1} m_{2}}{a_{2}}\left\{\bar{A} \, e_{1} \cos (l_{1}-2l_{2}+2 g_{1}-2 g_{2})-\bar{B} \, e_{2} \cos (l_{1}-2l_{2}+g_{1}-g_{2})\right\}\\+&\dfrac{m_{2} m_{3}}{a_{3}}\left\{\bar{A} \, e_{2} \cos (l_{2}-2l_{3}+2 g_{2}-2 g_{3})-\bar{B} \, e_{3} \cos (l_{2}-2l_{3}+g_{2}-g_{3})\right\} + \bar{C},
\end{align*}
in which $\bar{A}, \bar{B}$ are two functions of the semi major axes ratio $\alpha=\dfrac{a_{1}}{a_{2}}=\dfrac{a_{2}}{a_{3}}$,
\begin{align*}
&\bar{A}=\dfrac{3}{4}\alpha\, b_{3/2}^{1}(\alpha)-\dfrac{1}{2}\alpha\, b_{3/2}^{3}(\alpha)-\alpha^{2}\, b_{3/2}^{2}(\alpha)\\
&\bar{B}=\dfrac{3}{4}\alpha\, b_{1/2}^{1}(\alpha)+\dfrac{3}{2}\alpha\, b_{3/2}^{0}(\alpha)-\alpha^{2}\, b_{3/2}^{1}(\alpha)-\dfrac{1}{2}\alpha \, b_{3/2}^{2}(\alpha)\\
\end{align*}
and in which $\bar{C}$
is a certain function depending only on the masses and the semi major axes that will be ignored in the sequel. With $\alpha=\dfrac{1}{\sqrt[3]{4}}$, we find, in the $1+3$-body model, approximately $2 \bar{A}=2.3810, 2 \bar{B}=3.3764$\footnote{Note that, in \cite{DeSitter1909}, De Sitter took (for small eccentricities) $\epsilon_{i}=e_{i}/2, \, i=1,2,3$ instead of $e_{i}$ in the expression of $F_{res}$.}. The value of $\bar{A}$ is in consistence with De Sitter \cite{DeSitter1909} for the 4-body model. The value of $\bar{B}$ differs from the 4-body model (for which De Sitter had noted the numerical value $2 \bar{B}=0.964$, while we found $2 \bar{B}=0.8566$ by Maple) by a difference of $ \alpha^{-1/2}=\sqrt[3]{2}$, resulting from the indirect part of the perturbing function of the 4-body problem in the Joviancentric coordinates. 

The difference of the coefficient $\bar{B}$ in the two formalisms thus manifests the contribution of the indirect part of the perturbing function for the critical terms. To retain the simplicity of our analysis without diminishing its generality, in the sequel, we shall continue only with the assumptions $\bar{A}>0, \bar{B}>0$.

\section{Periodic Orbits of De Sitter} \label{Section: Periodic Orbits of De Sitter}
From Thm \ref{Theo: Poincare continuation}, we know that non-degenerate critical points give rise to periodic orbits of $F$ for small $\mu$. To look for such critical points, the particular form of $F_{res}$ (that it contains only cosines) motivates the proposal of Poincaré of the ansatz\footnote{In \cite{DeSitter1909}, De Sitter stated this as a necessity. This is not evident at all.} of collinearities of the three satellites when they simultaneously pass through their pericentres/apocentres (c.f. Figure \ref{Fig: collinearity}). We therefore only consider those solutions of $F_{Kep}$ passing through those points with 
$$(l_{1}, l_{2}, l_{3}, g_{1}, g_{2}, g_{3})=(0\hbox{ or }\pi,0\hbox{ or }\pi,0\hbox{ or }\pi,g_{2}\pm \pi,g_{2}, g_{2} \pm \pi),$$
thus 32 families of them. These families are not entirely different. Indeed, since $\delta_{1}=l_{1}-2 l_{2}, \delta_{2}=l_{2}-2 l_{3}$ does not vary along the resonance, leaving only the fast angle $l_{3}$ vary, the two families different only by $\pi$ in the $l_{3}$-argument give rise to the same family of the periodic orbits. We therefore obtained 16 families of periodic orbits labelled by 
$$(\delta_{1}, \delta_{2}, \eta_{1}:=g_{1}-g_{2}, \eta_{2}:=g_{2}-g_{3})=(\dfrac{\pi}{2} \pm \dfrac{\pi}{2}, \dfrac{\pi}{2} \pm \dfrac{\pi}{2}, \dfrac{\pi}{2} \pm \dfrac{\pi}{2}, \dfrac{\pi}{2} \pm \dfrac{\pi}{2}).$$
 We name these families respectively by $\mathcal{D}_{+,+,+,+}, \mathcal{D}_{+,+,+,-}\cdots$ etc., with subscript coincides with the ordered signs appearing in the $\pm$'s of the right hand side.

With this ansatz, a solution of $F_{Kep}+F_{res}$ is indeed periodic in some proper uniform-rotating frame, or in the reduced system by the $SO(2)$-rotational symmetry, if and only if the frequencies of the precessions of the pericentres are equal. When such a periodic solution is found, its normal dynamics, as well as that of its continuation for small $\mu$ can thus be determined by the corresponding evaluation of the Hessian of $F_{Kep}+F_{res}$.

\subsection{Delaunay Coordinates and Reduction of the Rotational Symmetry}
To effectively calculate the frequencies of the pericentres and to explicitly carry out the symplectic reduction of the rotation symmetry, we now introduce the canonical Delaunay coordinates 
$(L_i,l_i,G_i,g_i,H_i,h_i),i=1,2,3$, defining as the following:

\begin{equation*} 
\left\{
\begin{array}{ll}L_i=\mu_i \sqrt{M_i} \sqrt{a_i}   & \hbox{circular angular momentum}\\ l_i  &\hbox{mean anomaly}\\ G_i = L_i \sqrt{1-e_i^2} &\hbox{angular momentum} \\g_i &\hbox{argument of pericentre} 
\end{array}\right.
\end{equation*}
In which 
$\mu_{i}=\bar{m}_{i}, M_{i}=m_{0}$ in the $1+3$-body problem, and $\mu_{i}=m_{0} \bar{m}_{i}/(m_{0}+\mu \bar{m}_{i}), M_{i}=m_{0}+\mu \bar{m}_{i}$ in the 4-body problem.

\begin{figure}
\centering
\includegraphics[width=3in]{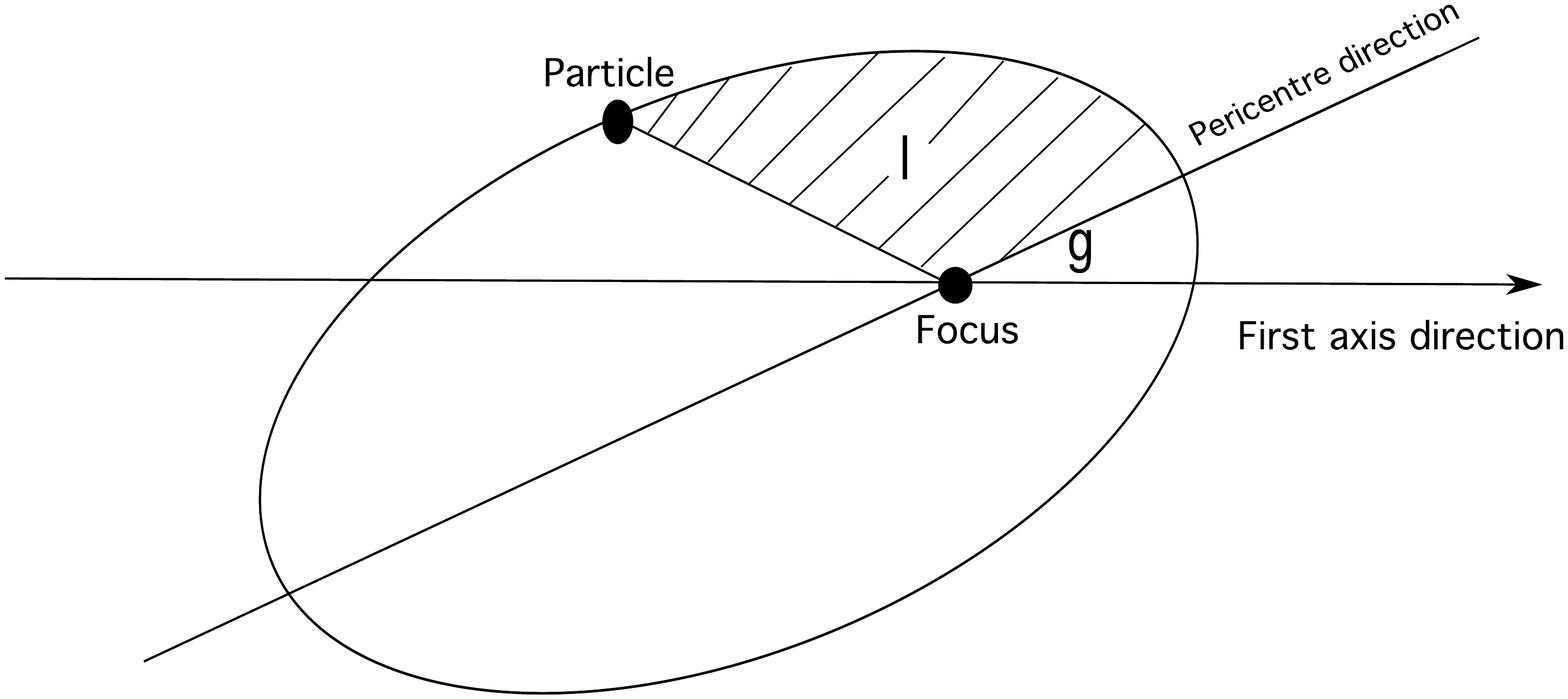}
\caption{Planar Delaunay Angles}
\end{figure}

To cope with our analysis near the $1:2:4$-resonance and to facilitate the reduction procedure, we modify these Delaunay coordinates to a set of symplectic coordinates $(D_{1}, D_{2}, D_{3}, d_{1}, d_{2}, d_{3}, Z_{1}, Z_{2}, Z_{3}, \eta_{1}, \eta_{2}, \eta_{3})$ in which 
\begin{equation*}
\left\{
\begin{array}{ll}
&D_{1}=L_{1}, \,\,\quad \quad \qquad \qquad \qquad \delta_{1}=l_{1}-2 l_{2} , \\
&D_{2}=2 L_{1}+L_{2}, \,\,\,\,\, \quad \qquad \qquad \delta_{2}=l_{2}-2 l_{3}, \\
&D_{3}=4 L_{1}+2 L_{2}+L_{3},\, \,\quad \qquad \delta_{3}=l_{3},\\
&Z_{1}=G_{1},  \,\,\, \quad \quad \quad \quad \qquad \qquad \eta_{1}=g_{1}-g_{2} , \\
&Z_{2}=G_{1}+G_{2}, \,\,\quad \quad \qquad \qquad \eta_{2}=g_{2}-g_{3}, \\
&Z_{3}=G_{1}+G_{2}+G_{3}, \, \quad \quad \qquad \eta_{3}=g_{3}.\\
\end{array}\right.
\end{equation*}
The coordinate $Z_{3}$ is just the total angular momentum; the rotational symmetry of the system translates into the independence of $F$ of $g_{3}$. The corresponding symplectic reduction procedure is thus achieved by fixing $Z_{3}$ and ignoring the variable $g_{3}$. 

\subsection{Elimination of the Fast Angle}\label{Subsection: Elimination procedure}
We suppose $0< a_{1} < a_{2} < a_{3}, 0<e_{1}, e_{2}, e_{3} < e^{\wedge}<1$ such that 
$$\dfrac{a_{1} (1+e^{\wedge})}{a_{2}(1-e^{\wedge})}<1, \dfrac{a_{2} (1+e^{\wedge})}{a_{3}(1-e^{\wedge})}<1$$
so that the three elliptic orbits are bounded away from each other for all time.

All these requirements determine an open subset $\mathcal{P}$ of the phase space.
The coordinates 
$$(D_{1}, D_{2}, D_{3}, \delta_{1}, \delta_{2}, \delta_{3}, Z_{1}, Z_{2}, Z_{3}, \eta_{1}, \eta_{2}, \eta_{3})$$ 
identifies $\mathcal{P}$ to a subset $\tilde{\mathcal{P}}$ of $\T^{6} \times \R^{6}$. The variables $\delta_{1}, \delta_{2}$ are critical (or semi-fast) arguments near the $4:2:1$-resonance, with, near this resonance, the only fast angle $\delta_{3}$. 

Let $(D_{1}^{0}, D_{2}^{0}, D_{3}^{0}) \in \R^{3}_{+}$ be so chosen such that the set
$$\{D_{1}=D_{1}^{0}, D_{2}=D_{2}^{0}, D_{3}=D_{3}^{0}\}$$
consists in $4:2:1$-resonant Keplerian motions in $\tilde{\mathcal{P}}$. For $C_{1}>0,$ We define $\tilde{N}$ to be the $C_{1} \mu$-neighborhood of the set $\{D_{1}=D_{1}^{0}, D_{2}=D_{2}^{0}, D_{3}=D_{3}^{0}\}$ in $\tilde{\mathcal{P}}$. We shall assume that $C_{1}$ is chosen large enough to allow further application of KAM theorems. We let $\check{N}$ to be the $C_{2}$-neighborhood of $\{D_{1}=D_{1}^{0}, D_{2}=D_{2}^{0}, D_{3}=D_{3}^{0}\}$ in $\tilde{\mathcal{P}}$ for some small enough $C_{2}$ so as to fixed a neighborhood for complex continuation of analytic functions. The small parameter $\mu$ is supposed to satisfy $C_{1} \mu < C_{2}$.
 
Let $T_{\C}=\C^6/\Z^6 \times \C^6$ and $T_s:=\{z \in T_{\C}: \exists \, z' \in \T^6 \times \R^6 \hbox{ s.t. } |z-z'| \le s\}$ be the $s$-neighborhood of $\T^6 \times \R^6:=\R^{6}/\Z^{6} \times \R^{6}$ in $T_{\C}$. Let $T_{\textbf{A},s}$ be the $s$-neighborhood of a set $\textbf{A} \subset \T^6 \times \R^6$ in $T_s$. 
The complex modulus of a transformation is the maximum of the complex moduli of its components. We use $| \cdot |$ to denote the modulus of either a function or a transformation. 

Being analytic functions on $\check{N}$, there exist $s>0$, such that $F_{Kep}$ and $F_{pert}$ extend to analytic functions on $T_{\check{N}, s}$.

For a function $f: T^{6} \times \R^{6} \to \R$, we define
$$\langle f \rangle_{\delta_{3}}=\dfrac{1}{2 \pi} \int_{0}^{2 \pi} f d \delta_{3}. $$

\begin{prop}\label{prop: elimination} There exists an $O(\mu)$-analytic symplectic transformation $\phi: \tilde{N} \to \phi(\tilde{N})$ such that
$$\phi^{*} F = F_{Kep}+F_{res}+F_{rem},$$
in which in $\tilde{N}$ the analytic functions
\begin{itemize}
\item$F_{res}=\langle F_{res} \rangle_{\delta_{3}}$ is the truncation at the first order of the eccentricities of the function $\langle F_{pert} \rangle_{\delta_{3}}$, and 
\item $F_{rem}=O(\mu e^{2})+O(\mu^{2})$.
\end{itemize}
\end{prop}
\begin{proof}  We search for an auxiliary Hamiltonian $\hat{H}$ whose time-1 map gives the transformation $\phi$. We have

$$\phi^*F=F_{Kep}+(F_{pert}+X_{\hat{H}} \cdot F_{Kep})+F^1_{comp,1},$$
in which $X_{\hat{H}}$ is seen as a derivation operator. Let $\bar{F}_{res}=\langle F_{pert} \rangle_{\delta_{3}}$ be the average of $F_{pert}$ over $\delta_{3}$, and
 $\widetilde{F}_{pert}=F_{pert}-\langle F_{pert}\rangle_{\delta_{3}}$ be the zero-average part of $F_{pert}$.

In the system $F_{Kep}$, since the frequencies $\nu_{Kep, 1}, \nu_{Kep,2}$ of $\delta_{1}$ and $\delta_{2}$ is of order $O(\mu)$ compared to the frequency $\nu_{Kep, 3}$ of $\delta_{3}$, we do not have to solve the exact cohomological equation 
$$\nu_{Kep,1} \partial_{\delta_1} \hat{H}+ \nu_{Kep,2} \partial_{\delta_2} \hat{H}+\nu_{Kep,3} \partial_{\delta_3} \hat{H}=\widetilde{F}_{pert};$$
instead, we just need $\hat{H}$ to solve the perturbed cohomological equation
$$\nu_{Kep,3} \partial_{\delta_3} \hat{H}=\widetilde{F}_{pert}.$$
We thus set
$$\hat{H}= \dfrac{1}{\nu_{Kep,3}}  \int_0^{\delta_{3}}  \widetilde{F}_{pert}\, d \delta_{3}$$
as long as $\nu_{Kep,3}  \neq 0$, which is indeed satisfied as a frequency of Keplerian elliptic motion. This amounts to proceed with a single frequency elimination for $\delta_{3}$.

We have 
 $$|\hat{H}| \le \hbox{Cst } \mu \, \hbox{   in   } T_{\check{N}, s} . $$
We obtain by Cauchy inequality that in $T_{\check{N},s-s_0}$, $|X_{\hat{H}}| \le \hbox{Cst } |\hat{H} | \le \hbox{Cst } \mu$ for some $0<s_0<s/2$. Shrinking from $T_{\check{N},s- s_0}$ to $T_{\check{N}^{**},s- s_0-s_1}$, where $\check{N}^{**}$ is an open subset of $\check{N}$, so that $\phi (T_{\check{N}^{**},s- s_0-s_1}) \subset T_{\check{N},s-s_0}$, with $s- s_0-s_1>0$. The time-1 map $\phi$ of  $X_H$ thus satisfies $|\phi-Id | \le \hbox{Cst } \mu$ in $T_{\check{N}^{**},s- s_0- s_1}$. The function $\phi^* F$ is analytic in $T_{\check{N}^{**},s- s_0-s_1}$.

The function $F$ is thus analytically conjugated to
$$\phi^* F =F_{Kep}+\bar{F}_{res} + F^1_{comp,1},$$ and $|F^1_{comp,1}|$ is of order $O(\mu^{2})$: indeed, analogously as in \cite{QuasiMotionPlanar}, the complementary part
 $$F_{comp,1}^1 = \int_0^1 (1-t) \phi_t^*(X_{\hat{H}}^2 \cdot F_{Kep})dt+ \int_0^1 \phi_t^*(X_{\hat{H}} \cdot F_{pert}) dt  - \nu_{Kep,2} \partial_{\delta_2} \hat{H}-\nu_{Kep,3} \partial_{\delta_3} \hat{H}$$
 satisfies
 $$|F_{comp,1}^1| \le \hbox{Cst } |X_{\hat{H}}| (|\tilde{F}_{pert}|+|F_{pert}|) + (\nu_{Kep,2}+\nu_{Kep,3}) | \hat{H} | \le \hbox{Cst } \mu^{2}.$$

To finish the proof, it is enough to expand $\bar{F}_{res}$ in powers of the eccentricities: $\bar{F}_{res}=F_{res}+O(\mu e^{2})$, in which $F_{res}$ only contains terms of first order in eccentricities.
\end{proof}

\begin{rem}\label{rem: elimination} With the same method, the elimination procedure can be further continued to eliminate the dependence of $\delta_{3}$ in higher order perturbations as well (c.f. \cite{QuasiMotionPlanar}). 
\end{rem}

The reduced system is thus determined by the Hamiltonian function
$$F_{Kep}(D_{1}, D_{2}, D_{3})+F_{res}(D_{1}, D_{2}, D_{3}, \delta_{1}, \delta_{2}, Z_{1}, Z_{2}, \eta_{1}, \eta_{2}; Z_{3})+o\Bigl(\mu e^{2}\Bigr)+o(\mu^{2})$$
defined on a subset of  $\T^{5} \times \R^{5}$, with $Z_{3}$ appearing as a parameter.

\subsection{Relative Frequencies of the Pericentres}
We call the frequencies 
\begin{align*}
&\nu_{1}=\dfrac{\partial F_{res}}{\partial Z_{1}}=-\dfrac{2 \sqrt[3]{2} \bar{A} \cos(\delta_{1}+2 \eta_{1}) e_{2} m_{2} +2 \bar{B} \cos(\delta_{1}+ \eta_{1}) e_{1} m_{1}- \sqrt[3]{2} \bar{A} \cos(\delta_{2}+2 \eta_{2}) e_{1} m_{3} }{\sqrt{m_{0}} e_{1} e_{2}},\\ &\nu_{2}=\dfrac{\partial F_{res}}{\partial Z_{2}}=\dfrac{2 \bar{B} \cos(\delta_{1}+ \eta_{1}) e_{3} m_{1} - 2 \sqrt[3]{2} \bar{A} \cos(\delta_{2}+ 2 \eta_{2}) e_{3} m_{3}- 2 \bar{B} \cos(\delta_{2}+ \eta_{2}) e_{2} m_{2}}{\sqrt{m_{0}} e_{2} e_{3}}
\end{align*}
of $\eta_{1}=g_{1}-g_{2}$ and $\eta_{2}=g_{2}-g_{3}$ \emph{relative frequencies} of the pericentres in $F_{Kep}+ F_{res}$. They appear as differences of the three quantities
$$\nu_{g_{1}}=-\dfrac{2 \sqrt[3]{2} \bar{A} \cos(\delta_{1}+2 \eta_{1})  m_{2}}{\sqrt{m_{0}} e_{1} }, \nu_{g_{2}}=\dfrac{2 \bar{B} \cos(\delta_{1}+ \eta_{1})  m_{1} - 2 \sqrt[3]{2} \bar{A} \cos(\delta_{2}+ 2 \eta_{2})  m_{3}}{\sqrt{m_{0}} e_{2} }, \nu_{g_{3}}=\dfrac{2 \bar{B} \cos(\delta_{2}+ \eta_{2})  m_{2}}{\sqrt{m_{0}} e_{3}},$$
which are frequencies of $g_{1}, g_{2}$ and $g_{3}$ respectively.
 
To obtain periodic solutions of $F_{Kep}+ F_{res}$ after reduction of the rotational symmetry, we have to ask $\nu_{1}=\nu_{2}=0$. A simple analysis \cite[pp. 10-12]{DeSitter1909} with the signs of $\nu_{g_{1}}, \nu_{g_{2}}, \nu_{g_{3}}$ shows that if we do not propose additional conditions on the masses, then this is only possible for families $D_{-,-,+,+}$ (Case (6) of \cite{DeSitter1909}) and $D_{+,+,+,+}$ (Case (16) of \cite{DeSitter1909}). 
Otherwise, with $\bar{Q}=\sqrt[3]{2} \bar{A}  \bar{m}_{3}-2 \bar{B}  \bar{m}_{1}$, we see that $\bar{Q}>0$, $D_{-,-,-,+}$ (Case (2) of \cite{DeSitter1909}), $D_{+,+,-,+}$ (Case (12) of \cite{DeSitter1909}), and for $\bar{Q}<0$, $D_{-,+,+,-}$ (Case (7) of \cite{DeSitter1909}), $D_{+,-,+,-}$ (Case (13) of \cite{DeSitter1909},\,see \cite[p.10]{DeSitter1909}) are also possible for the nullity of $\nu_{1}$ and $\nu_{2}$. In any of these cases, we obtain a one-parameter family of eccentricities $(e_{1}, e_{2}, e_{3})$ (parametrized by one of the eccentricities, \emph{e.g.} by $e_{2}$), such that the two relative frequencies $\nu_{1}, \nu_{2}$ are zero.

\subsection{Continuation of Periodic Orbits}\label{Subsection: Continuation}
The existence of the corresponding families of periodic solutions of $F_{Kep}+F_{res}$ after reduction of the rotational symmetry thus follows from the above-mentioned families of eccentricities corresponding to the nullity of $\nu_{1}, \nu_{2}$: The frequencies of $\delta_{1}, \delta_{2}, \eta_{1}, \eta_{2}$ are all zero, leaving only the frequency of the angle $\delta_{3}$ non-zero.

We now verify that Thm \ref{Theo: Poincare continuation} is applicable in all these cases. The Hessian of $F_{Kep}$ with respect to $L_{1}, L_{2}, L_{3}$ is a non-degenerate diagonal matrix. We therefore only need to verify that the Hessian of $F_{res}$ with respect to $\delta_{1}, \delta_{2}, Z_{1}, Z_{2}, \eta_{1}, \eta_{2}$ is non-degenerate. For this Hessian matrix, we summarize the results in the following table
\begin{center}
\begin{tabular}{|l|c|c|}
  \hline
  Case & Sign of $\bar{Q}$ & Sign of the determinant of the Hessian matrix  \\
  \hline
  $D_{-,-,+,+}$ & irrelevant & positive \\
  $D_{+,+,+,+}$ & irrelevant & positive \\
  $D_{-,-,-,+}$ & positive &  negative \\
  $D_{+,+,-,+}$ & positive & negative \\
  $D_{-,+,+,-}$  & negative &  negative\\
  $D_{+,-,+,-}$ & negative & negative \\
  \hline
\end{tabular}
\end{center}
The additional perturbation is of the order $O(\mu e^{2})+O(\mu^{2})$. We thus obtain
\begin{prop}
The normalized $m_{0}, \bar{m}_{1}, \bar{m}_{2}, \bar{m}_{3}$ being fixed, after reduction of the rotational symmetry, for sufficiently small $\mu, e$, each of families $D_{-,-,+,+}$, $D_{+,+,+,+}$, $D_{-,-,-,+}$, $D_{+,+,-,+}$, $D_{-,+,+,-}$, $D_{+,-,+,-}$ of the periodic orbits of $F_{Kep}+F_{res}$, provided that their existence is permitted by the fixed value of $\bar{m}_{1}$ and $\bar{m}_{3}$, can be continued to families of periodic orbits of $F$.
\end{prop}

\begin{rem}In \cite{DeSitter1909}, the required non-degeneracy condition is not verified before the analysis of their stability. The non-degeneracy at the linearly stable periodic orbit is verified \emph{a posteriori} by showing the non-degeneracy of its normal dynamics. 
\end{rem}

\subsection{Normal Dynamics/Linear Stability of the Periodic Orbits}
We now determine the normal dynamics and the linear stability of these periodic orbits. For this purpose, an analysis for those periodic orbits in $F_{Kep}+F_{res}$ suffices. For this, it is enough to fix $D_{3}$ and reduce the system by the $SO(2)$-symmetry of shifting $\delta_{3}$, which reduce the periodic orbits to equilibria $E_{-,-,+,+}$ and $E_{+,+,+,+}$ respectively in the reduced system. 

We aim to calculate the eigenvalues of the linearization at the corresponding equilibra by using the Hessian of $F_{Kep}+F_{res}$ in coordinates $(D_{1}, D_{2}, \delta_{1}, \delta_{2}, Z_{1}, Z_{2}, \eta_{1}, \eta_{2})$ with small $\mu$.  Since $F_{Kep}$ and $F_{res}$ do not appear in the same magnitude of $\mu$, it is enough to calculate the corresponding linearization matrix $\mathcal{L}$ from the Hessian $\mathcal{H}_{1}$ of $F_{Kep}$ with respect to $D_{1}, D_{2}$ and the Hessian $\mathcal{H}_{2}$ of $F_{res}$ with respect to all other variables by
$$\mathcal{L}=J \cdot Diag\{\mathcal{H}_{1}, \mathcal{H}_{2}\}$$
in which
\begin{center}
\begin{math}
J=\begin{pmatrix}
0 &0 & -1 & 0 & 0 & 0 & 0 & 0\\
0 & 0 & 0  & -1 & 0 & 0 & 0 & 0\\
1 & 0 & 0 & 0 & 0 & 0 & 0 & 0\\
0 & 1 & 0 & 0 & 0 & 0 & 0 & 0 \\
0 &0 & 0  & 0 & 0 & 0 & -1 & 0\\
0 & 0 & 0  & 0 & 0 & 0 & 0 & -1\\
0 & 0 & 0  & 0 & 1 & 0 & 0 & 0\\
0 & 0 & 0 & 0  & 0 & 1 & 0  & 0 \\
\end{pmatrix}.
\end{math}
\end{center}

We thus find that the linearization of the reduced flow at this point is a matrix of the form
\begin{center}
\begin{math}
\begin{pmatrix}
0 &0 & \mu^{2} e & 0 & 0 & 0 & \mu^{2} e & 0\\
0 & 0 & 0  & \mu^{2} e & 0 & 0 & 0 & \mu^{2} e\\
\mu^{-1} & \mu^{-1} & 0 & 0 & 0 & 0 & 0 & 0\\
\mu^{-1} & \mu^{-1} & 0 & 0 & 0 & 0 & 0 & 0 \\
0 &0 & \mu^{2} e  & 0 & 0 & 0 & \mu^{2} e & 0\\
0 & 0 & 0  & \mu^{2} e & 0 & 0 & 0 & \mu^{2} e\\
0 & 0 & 0  & 0 & e^{-3} & e^{-3} & 0 & 0\\
0 & 0 & 0 & 0  & e^{-3} & e^{-3} & 0  & 0 \\
\end{pmatrix}.
\end{math}
\end{center}
in which all entries are only indicated by their orders in $\mu$ and $e$.

We suppose that $e$ is positive, non-zero and small enough. In this section, we aim to analyze whether the eigenvalues of this matrix is distinct and purely imaginary. For this purpose, it is enough to only argue with the small parameter $\mu$ and we may leave the parameter $e$ aside for a moment. In Section \ref{Sec: KAM near periodic orbit}, both of the orders of the frequencies in small parameters $\mu$ and $e$ will be needed when we apply KAM theory to find librating quasi-periodic orbits.

Due to the appearance of entries having negative powers of the small parameter $\mu$, the calculation of the eigenvalues of this matrix is not convenient. This deficit is however found by De Sitter to be avoidable if we appropriately rescale the coordinates while preserving their symplecticity. 

We now treat $\delta_{i}, \,\eta_{i}, \, i=1,2$ as periodic functions on $\R$ instead of functions on $\T$, as otherwise we cannot continuously rescale them. Let $c_{0}=c_{0}(\mu)$  be the rescaling factor to be determined. We take $(\check{D}_{1}, \check{\delta}_{1}, \check{D}_{2}, \check{\delta}_{2},\check{Z}_{1}, \check{\eta}_{1}, \check{Z}_{2}, \check{\eta}_{2})$  as new coordinates with
\begin{equation*}
\left\{
\begin{array}{ll}
&D_{1}=\sqrt{c_{0}}\,\check{D}_{1}, \,\,\,\,\, \quad \qquad \qquad  \delta_{1}=\check{\delta}_{1}/\sqrt{c_{0}}, \\
&D_{2}=\sqrt{c_{0}}\,\check{D}_{2}, \,\,\,\,\, \quad \qquad \qquad \delta_{2}=\check{\delta}_{2}/\sqrt{c_{0}}, \\
&Z_{1}=\sqrt{c_{0}}\,\check{Z}_{1},  \,\,\quad \quad \qquad \qquad  \eta_{1}=\check{\eta}_{1}/\sqrt{c_{0}}, \\
&Z_{2}=\sqrt{c_{0}}\,\check{Z}_{2}, \,\,\quad \quad \qquad \qquad \eta_{2}=\check{\eta}_{2}/\sqrt{c_{0}}.
\end{array}\right.
\end{equation*}

By choosing \emph{a posteriori} $c_{0}=\sqrt{\mu^{3}}$, in these new coordinates, the corresponding linearization matrix takes the form (only the orders in $\mu$ are expressed)
\begin{center}
\begin{math}
\begin{pmatrix}
0 &0 & \sqrt{\mu}  & 0 & 0 & 0 & \sqrt{\mu} & 0\\
0 & 0 & 0  & \sqrt{\mu}  & 0 & 0 & 0 & \sqrt{\mu} \\
\sqrt{\mu}  & \sqrt{\mu}  & 0 & 0 & 0 & 0 & 0 & 0\\
\sqrt{\mu}  & \sqrt{\mu}  & 0 & 0 & 0 & 0 &  0 & 0 \\
0 &0 & \sqrt{\mu}  & 0 & 0 & 0 & \sqrt{\mu} & 0\\
0 & 0 & 0  & \sqrt{\mu}  & 0 & 0 & 0 & \sqrt{\mu} \\
0 & 0 & 0  & 0 & \sqrt{\mu^{3}} &\sqrt{\mu^{3}} & 0 & 0\\
0 & 0 & 0 & 0  & \sqrt{\mu^{3}} & \sqrt{\mu^{3}} & 0  & 0 \\
\end{pmatrix},
\end{math}
\end{center}
with every entry containing a common factor $\sqrt{\mu}$. This factor being ruled out, we arrive at a matrix having the form (again, only the orders in $\mu$ are expressed)
\begin{center}
\begin{math}
\begin{pmatrix}
0 &0 & 1  & 0 & 0 & 0 & 1 & 0\\
0 & 0 & 0  & 1  & 0 & 0 & 0 & 1 \\
1  & 1  & 0 & 0 & 0 & 0 & 0 & 0\\
1  & 1  & 0 & 0 & 0 & 0 &  0 & 0 \\
0 &0 & 1  & 0 & 0 & 0 & 1 & 0\\
0 & 0 & 0  &  1 & 0 & 0 & 0 & 1\\
0 & 0 & 0  & 0 & \mu & \mu & 0 & 0\\
0 & 0 & 0 & 0  & \mu & \mu & 0  & 0 \\
\end{pmatrix},
\end{math}
\end{center}
we observe that by letting $\mu$ tend to zero, 4 eigenvalues of this matrix tends to 0, with 4 other eigenvalues tending to the eigenvalues of the upper-left $4 \times 4$-block, and thus their square powers satisfy the quadratic equation
\begin{equation} \label{eq: 1st quadratic}
x^{2}-(A_{11} s+A_{22} s') x + (A_{11} A_{22}- A_{12}^{2}) s s'=0,
\end{equation}
in which 
$$A_{11}=\mu^{-1/2} \dfrac{\partial^{2} F_{Kep}}{\partial \check{D}_{1}^{2}}, A_{12}=\mu^{-1/2} \dfrac{\partial^{2} F_{Kep}}{\partial \check{D}_{1} \partial \check{D}_{2}}, A_{22}=\mu^{-1/2}  \dfrac{\partial^{2} F_{Kep}}{\partial \check{D}_{2}^{2}}, s=-\mu^{-1/2}  \dfrac{\partial^{2} F_{Kep}}{\partial \check{\delta}_{1}^{2}}, s'=-\mu^{-1/2}  \dfrac{\partial^{2} F_{Kep}}{\partial \check{\delta}_{2}^{2}}.$$

The determination of the other 4 eigenvalues is achieved by a proper rescaling of the unknown variable $\lambda$ in the characteristic polynomial of this matrix. Indeed, in the corresponding determinant, if we replacing $\lambda$ by $\sqrt{\mu}\, \lambda'$, we find that in the lower-right $4 \times 4$-block, the factor $\sqrt{\mu}$ appearing in the $\sqrt{\mu} \,\lambda'$'s can be ruled out by dividing out the factor $\sqrt{\mu}$ in the last two rows and (then) the fifth and sixth columns. By letting $\mu$ tend to zero, we obtain a determinant, which (under the condition $A_{11} A_{22}- A_{12}^{2}>0,\,s s' \neq 0$, which could be verified easily) gives rise to a quadratic equation in ${\lambda'}^{2}$ (which is found to have the same form of Eq. \ref{eq: 1st quadratic}):
\begin{equation} \label{eq: 2nd quadratic}
x^{2}-(B_{11} \sigma+B_{22} \sigma') x + (B_{11} B_{22}-B_{12}^{2}) \sigma \sigma'=0,
\end{equation}
in which 
$$B_{11}=\mu^{-3/2} \dfrac{\partial^{2} F_{res}}{\partial \check{Z}_{1}^{2}}, B_{12}=\mu^{-3/2} \dfrac{\partial^{2} F_{res}}{\partial \check{Z}_{1} \partial \check{Z}_{2}}, B_{22}=\mu^{-3/2} \dfrac{\partial^{2} F_{res}}{\partial \check{Z}_{2}^{2}}$$
and
$$ \sigma=-\mu^{-1/2} \dfrac{\partial^{2} F_{res}}{\partial \check{\eta}_{1}^{2}}+\mu^{-1/2} \left(\dfrac{\partial^{2} F_{res}}{\partial \check{\delta}_{1} \partial \check{\eta}_{1}}\right)\Big{/}\dfrac{\partial^{2} F_{res}}{\partial \check{\delta}_{1}^{2}}, \sigma'=-\mu^{-1/2} \dfrac{\partial^{2} F_{res}}{\partial \check{\eta}_{2}^{2}}+\mu^{-1/2} \left(\dfrac{\partial^{2} F_{res}}{\partial \check{\delta}_{2} \partial \check{\eta}_{2}}\right)\Big{/}\dfrac{\partial^{2} F_{res}}{\partial \check{\delta}_{2}^{2}}.$$

As such, we have obtained two quadratic equations. For small $\mu$ and small $e$, the corresponding periodic solutions are linear stable if both of these quadratic equations only have distinct\footnote{Otherwise, property being not open, a linearly stable periodic orbit is not necessarily continued to linearly stable periodic orbits for small parameters. De Sitter seems to have overlooked this point, though it does not change the result.} real, negative roots.

The discriminant of Eq. \ref{eq: 1st quadratic} is 
$$(A_{11} s+A_{22} s')^{2}-4 (A_{11} A_{22}- A_{12}^{2}) s s'=(A_{11} s-A_{22} s')^{2}+4 A_{12}^{2} s s'.$$
Since 
$$A_{11}<0, A_{22}<0, A_{11} A_{22}- A_{12}^{2}>0,$$
Eq. \ref{eq: 1st quadratic} has two distinct real negative roots if and only if 
$$s>0, \, s'>0.$$
As a result, by evaluating the values of $s$ and $s'$ at the corresponding points, we find that $E_{+,+,+,+}$, $E_{+,+,-,+}$, $E_{+,-,+,-}$ are unstable.

Having verified that $B_{11} B_{22}-B_{12}^{2}>0$ is satisfied in all the other cases (\emph{i.e.} for $E_{-,-,-,+}$, $E_{-,+,+,-}$, $E_{-,-,+,+}$), the analysis of Eq.  \ref{eq: 2nd quadratic} is similar. We thus find the following necessary and sufficient condition for Eq. \ref{eq: 2nd quadratic} to have two distinct real negative roots:
$$\sigma>0, \sigma'>0.$$
By evaluating the values of $
\sigma$ and $\sigma'$ at the corresponding points, we thus find that $E_{+,-,+,+}$, $E_{+,+,-,-}$ are unstable, and $E_{-,-,+,+}$ is the only linearly stable family. In terms of periodic orbits, we have

\begin{prop}(De Sitter, \cite{DeSitter1909}) $D_{-,-,+,+}$ is the only linearly stable family (parametrized by the small parameter $\mu$ and the eccentricity $e_{2}$) of periodic orbits in the $SO(2)$-reduced system of $F_{Kep}+F_{res}$ which can be continued to periodic orbits of $F$.
\end{prop}

For $D_{-,-,+,+}$, we find from the condition $\nu_{1}=\nu_{2}=0$ that
$$e_{1}=\dfrac{2 \cdot 2^{5/6} \bar{A} \bar{m}_{2} e_{2}}{2 \sqrt{2} \bar{B} \bar{m}_{1}+2^{5/6} \bar{A} \bar{m}_{3}},\,e_{3}=\dfrac{\sqrt{2} \bar{B} \bar{m}_{2} e_{2}}{2 \sqrt{2} \bar{B} \bar{m}_{1}+2^{5/6} \bar{A} \bar{m}_{3}}.$$

The continuation of these linearly stable periodic orbits thus provide a mathematical explanation of the real evolution of the system Jupiter-Io-Europa-Ganymede, which was the base of De Sitter's theory of the Galilean satellites \cite{DeSitter1925}, \cite{DeSitter1931}.

\begin{question}
Is the family $D_{-,-,+,+}$ also linearly stable in the spatial problem?
\end{question}

\section{KAM theorem}\label{Section: KAM theorem}
In order to establish the existence of librating quasi-periodic orbits, we shall make use of a general KAM theorem with parameters according to J. F\'ejoz \cite{FejozStability},  \cite{FejozMoser}, \cite{FejozHabilitation}. We note that a similar KAM theorem has also been established in \cite{BHT}.

\subsection{Hypothetical Conjugacy Theorem}

For $p \ge 1$ and $q \ge 0$, consider the phase space $\R^{p} \times
\T^p \times \R^q \times \R^q = \{(I,\theta,x,y)\}$ endowed with the
standard symplectic form $d I \wedge d \theta + d x \wedge d y$. All
mappings are assumed to be analytic except when explicitly mentioned
otherwise.

Let $\delta > 0$, $q'\in\{0,...,q\}$, $q''=q-q'$, $\varpi \in \R^p$,
and $\beta \in \R^q$. Let $B_{\delta}^{p+2 q}$ be the $(p+2 q)$-dimensional closed ball with radius $\delta$ centered at the origin in $\R^{p+2q}$, and $N_{\varpi,\beta}= N_{\varpi,\beta}
(\delta,q')$ be the space of Hamiltonians $N \in C^{\omega}(\T^p
\times B_{\delta}^{p+2q},\R)$ of the form
$$N= c + \langle \varpi,I \rangle + \sum_{j=1}^{q'}
\beta_j (x_j^2+ y_j^2) + \sum_{j=q'+1}^{q} \beta_j (x_j^2 -
y_j^2)+\langle A_1(\theta), I \otimes I \rangle + \langle A_2(\theta),
I \otimes Z \rangle + O_3(I,Z),
$$ 
with $c\in \R$, $A_1 \in C^{\omega}(\T^p, \R^{p} \otimes \R^p), A_2
\in C^{\omega}(\T^p,\R^p \otimes \R^{2q})$ and $Z = (x,y)$. The
isotropic torus $\T^p \times \{0\} \times \{0\}$ is an invariant
$\varpi$-quasi-periodic torus of $N$, and its normal dynamics is
elliptic, hyperbolic, or a mixture of both types, with Floquet
exponents $\beta$. The definitions of tensor operations can be found in e.g. \cite[p.62]{FejozStability}.

Let $\bar{\gamma} > 0$ and $\bar{\tau} > p-1$, $|\cdot|$ be the $\ell^2$-norm on
$\Z^{p}$. Let
$HD_{\bar\gamma,\bar\tau}=HD_{\bar{\gamma},\bar{\tau}}(p,q',q'')$ be the set of vectors $(\varpi,\beta)$ satisfying the
following homogeneous Diophantine conditions:
$$|k \cdot \varpi + l' \cdot \beta'| \geq \bar\gamma
(|k|^{\bar\tau}+1)^{-1}$$ for all $k\in \Z^p \setminus \{0\}$ and $l'
\in \Z^{q'}$ with $|l'_1| + \cdots + |l'_{q'}| \leq 2$. We have
denoted $(\beta_1,...,\beta_{q'})$ by $\beta'$. {Let $\|\cdot\|_{s}$ be the $s$-analytic norm of an analytic function, \emph{i.e.}, the supremum norm of its analytic extension to the $s$-neighborhood of its (real) domain in the complexified space $\C^{p} \times \C^{p}/\Z^{p}$.}

\begin{theo} \label{KAM} Let $(\varpi^o,\beta^o) \in
  HD_{\bar\gamma,\bar\tau}$ and $N^o \in N_{\varpi^o,\beta^o}$. For
 some $d>0$ small enough,  there exists $\varepsilon>0$ such that for every
  Hamiltonian $N' \in C^\omega(\T^p \times B_\delta^{p+2q})$ such that
  $$\|N'-N^o\|_d \leq \varepsilon,$$
 there exists a vector $(\varpi,\beta)$ satisfying the following
  properties:
  \begin{itemize}
  \item the map $N' \mapsto (\varpi,\beta)$ is of class $C^\infty$ and
    is $\varepsilon$-close to $(\varpi^o,\beta^o)$ in the
    $C^\infty$-topology;
  \item if $(\varpi,\beta) \in HD_{\bar\gamma,\bar\tau}$, $N'$ is 
    symplectically smoothly conjugate to a Hamiltonian $N \in
    N_{\varpi,\beta}$. 
  \end{itemize}
 Moreover, $\varepsilon$ can be chosen of the form $\hbox{Cst} \, \bar\gamma^{k}$
  (for some $\hbox{Cst}>0$, $k\geq 1$) when $\bar\gamma$ is small enough.
\end{theo}

Since analytic functions are $C^{\infty}$, this theorem directly follow from the corresponding ``Hypothetical Conjugacy Theorem'' of~\cite{FejozStability}. The fact that the aformentioned symplectic conjugation is actually analytic will appear in \cite{FejozMoser}.

\subsection{An Iso-chronic KAM theorem}

We now assume that the Hamiltonians $N^o=N^o_\iota$ and $N' = N'_\iota$
depend analytically on some parameter
$\iota \in B_1^{p+q}$. Recall that, for each $\iota$, $N^o_\iota$ is
of the form
\small
$$N^o_\iota= c^o_\iota + \langle \varpi^o_\iota,I \rangle + \sum_{j=1}^{q'}
\beta^o_{\iota,j} (x_j^2+ y_j^2) + \sum_{j=q'+1}^{q} \beta^o_{\iota,j}
(x_j^2 - y_j^2)+\langle A_{\iota,1}(\theta), I \otimes I \rangle +
\langle A_{\iota,2} (\theta), I \otimes Z \rangle + O_3(I,Z).$$\normalsize
Thm~\ref{KAM} can be applied to $N^o_\iota$ and $N'_\iota$ for each
$\iota$.  We will now add some non-degeneracy condition to the hypotheses of Thm \ref{KAM}, which ensures that
``$(\varpi_\iota,\beta_\iota) \in HD_{\bar\gamma,\bar\tau}$'' actually
occurs often in the set of parameters.

Denote by 
$$HD^o = \left\{(\varpi^o_\iota,\beta^o_\iota) \in
  HD_{\bar\gamma,\bar\tau}: \, \iota \in 
  B_{1/2}^{p+q}\right\}$$ the set of ``accessible''
$(\bar\gamma,\bar\tau)$-Diophantine unperturbed frequencies. The
parameter is restricted to a smaller ball so as to avoid boundary
problems.

\begin{cor}[Iso-chronic KAM theorem] \label{cor: isochron nondeg} Assume the
  map
  $$B_1^{p+q} \rightarrow \R^{p+q}, \quad \iota \mapsto
  (\varpi_\iota^o,\beta_\iota^o)$$ is a diffeomorphism onto its image.
  If $\varepsilon$ is small enough and if $\|N'_\iota-N^o_\iota\|_{d} <
  \varepsilon$ for each $\iota$, then for every $(\varpi,\beta) \in HD^o$ there exists a unique $\iota \in
  B^{p+q}_1$ such that $N'_\iota$ is symplectically conjugate to some
  $N \in N_{\varpi,\beta}$. Moreover, there exists $\bar{\gamma}>0, \bar{\tau}>p-1$, such that the set
  $$\{\iota \in B_{1/2}^{p+q}: \, (\varpi_\iota,\beta_\iota) \in
  HD^o\}$$ has positive Lebesgue measure.
\end{cor}

\begin{cor} [Proper-degenerate Iso-chronic KAM theorem] \label{cor: proper deg isochron nondeg}

  Assume 
  $$(\varpi_\iota^{o},\beta_\iota^{o})=(\varpi_\iota^{o,1},\beta_\iota^{o,1}, \varpi_\iota^{o,2},\beta_\iota^{o,2}, \cdots, \varpi_\iota^{o,\bar{n}},\beta_\iota^{o,\bar{n}} )$$ such that for $i_{j}=\hbox{dim} (\varpi_\iota^{o,j},\beta_\iota^{o,j})$, $\sum i_{j}=p+q$ and such that for any $j =1,2,\cdots, \bar{n}$, the frequencies involved in $(\varpi_\iota^{o,j},\beta_\iota^{o,j})$ appear in the same magnitude of some small quantity $\epsilon_{j}$. For small enough $1>>\epsilon_{1}>> \cdots >> \epsilon_{\tilde{n}}>0$ if
  the
  maps
  $$\bar{\nu}_{i}: B_1^{i_{j}} \rightarrow \R^{i_{j}}, \quad \iota \mapsto
  (\varpi_\iota^{o,j},\beta_\iota^{o,j})$$ are diffeomorphisms onto their images.
  then the
  map
  $$\bar{\nu}_{i}: B_1^{p+q} \rightarrow \R^{p+q}, \quad \iota \mapsto
  (\varpi_\iota^{o},\beta_\iota^{o})$$ is a diffeomorphism onto their images. 
  
  Consequently, we may choose $\varepsilon=\hbox{Cst} \epsilon_{1}^{N} \bar{\gamma}^{k}$ for large enough integer numbers $N$ and $k$.
\end{cor}

The proof goes in the same way as \cite[Example-Condition 5.3]{QuasiLunar}, in which it is enough to replace the action variables by general parameters, and the Hessian matrix by the corresponding Jacobian matrix. Again, it is enough to observe that the determinant of the block diagonal matrix with blocks representing the Jacobian matrix of the mapping $\bar{\nu}_{i}$ dominates the rest terms in the expression of the determinant of the Jacobian matrix of $\bar{\nu}$, as in the same spirit of \cite[Section 6.3.2]{ArnoldEncyclopedia}.

\section{Invariant KAM Tori around the Linearly Stable Periodic Orbits of De Sitter} \label{Sec: KAM near periodic orbit}

Theoretically, a better mathematical theory for the Galilean satellites could be given by the quasi-periodic orbits on the possibly-existing invariant KAM tori around this linearly stable periodic orbit, since as noticed by De Sitter himself \cite{DeSitter1909}, the $1:2:4$ resonance is only satisfied roughly by the inner three Galilean satellites. In this section, we shall explore some extensions of De Sitter's study by application of KAM theory.

We have shown that for $0<\mu<<1$ and $0<e<<1$, the elliptic equilibrium $E_{-,-,+,+}$ of the $SO(2) \times SO(2)$-reduced system $F_{Kep}+F_{res}$ is non-degenerate, with normal frequencies appearing at different orders in small quantities $\mu$ and $e$. 

We list the orders of these frequencies:
\begin{enumerate}
\item Frequency of the periodic orbit, which is the Keplerian frequency $\nu_{Kep, 3}$ of $\delta_{3}$. It is non-zero and of order $1$.

\item Elliptic normal frequencies:  It is direct to deduce from the dependence of the coefficients in Eqs. \ref{eq: 1st quadratic} and \ref{eq: 2nd quadratic}, that for $0<\mu<<e<<1$, two of them (denoted by $\nu_{n, 1}, \nu_{n, 2}$) are of order $\sqrt{\mu} \sqrt{e}$, the other two of them (denoted by $\nu_{n, 3}, \nu_{n, 4}$) are of order $\mu e^{-1}$.
\end{enumerate}

With the hypothesis $0<\mu<<e<<1$, the two normal frequencies are small in their orders compared to the frequency of the periodic orbit: consequently $\nu_{n, 1}, \nu_{n, 2}$ is of smaller order compared to $\nu_{n, 3}, \nu_{n, 4}$. 

To apply KAM theorems, we have to build higher order normal forms to control the smallness of the perturbation. We therefore consecutively eliminate the angle $\delta_{3}$ from the remainder $F_{rem}$ to have the higher-order remainder to be of the order $O(\mu e^{N}+\mu^{N})$ for any prescribed large enough $N$ by the same method as Prop \ref{prop: elimination}. The system $F_{Kep}+F_{res}+F_{rem}$ is thus analytically conjugate to 
$F_{Kep}+F_{res}^{N}+F_{rem}^{N}$
in which 
\begin{itemize}
\item $F_{res}^{N}-F_{res}=O(\mu e^{2}+\mu^{2} )$,
\item $F_{res}^{N}$ is independent of $l_{3}$,
\item For $N=2$, $F_{res}^{2}=F_{res}$, and
\item $F_{rem}^{N}$ is of order $O(\mu e^{N})+O(\mu^{N})$.
\end{itemize}

We now analyze the dynamics of $F_{Kep}+F_{res}^{N}$. After being symplectically reduced by the rotational $SO(2)$-symmetry, and the $SO(2)$-symmetry of shifting $l_{3}$, the reduced system is an $O(\mu^{2}+\mu e^{2})$-perturbation of the reduced system of $F_{Kep}+F_{res}$. Since the equilibrium point $E_{-,-,+,+}$ of the latter is non-degenerate, this equilibrium point continues to exist for small enough $\mu$ and $e$, and give rise to an non-degenerate normally elliptic equilibrium $E^{N}_{-,-,+,+}$ of the reduced system of $F_{Kep}+F_{res}^{N}$, which further give rise to a family of normally elliptic periodic orbits $D^{N}_{-,-,+,+}$ of $F_{Kep}+F_{res}^{N}$ (with only the rotational $SO(2)$-symmetry reduced). We thus reach the following 

\begin{lem} For any $N$ and small enough $\mu$ and $e$, there exists a family of normally elliptic periodic orbits $D^{N}_{-,-,+,+}$ of $F_{Kep}+F_{res}^{N}$ (reduced only by the rotational $SO(2)$-symmetry) continuing the family $D^{N}_{-,-,+,+}$ of $F_{Kep}+F_{res}$.
\end{lem} 

We introduce another small parameter $\bar{\delta}$, which indicates the distance of a point to $E^{N}_{-,-,+,+}$ in the phase space of the fully reduced system. We develop the fully reduced system of $F_{Kep}+F_{res}^{N}$ at $E^{N}_{-,-,+,+}$ so that 
$$F_{Kep}+F_{res}^{N}=F_{Kep}+F_{res,l}^{N}+O(\mu^{2} e \bar{\delta}^{3})$$
in which $F_{res, l}^{N}$ is the corresponding quadratic part whose Hamiltonian flow is the linearized system at  $E^{N}_{-,-,+,+}$. It is thus of the order $\mu^{2} e \bar{\delta}^{2}$.

We are thus in a situation to perturb the integrable system $F_{Kep}+F_{res,l}^{N}$ by some perturbation of the order $O(\mu^{2} e^{N})+O(\mu^{N})+O(\mu^{2} e \bar{\delta}^{3})$. We impose that 
$$0<< \mu << e<<1, \hbox{ and } 0<\bar{\delta}<<1 \hbox{ small enough}.$$
The frequencies of those Lagrangian tori of $F_{Kep}+F_{res,l}^{N}$ is dominated by the tangential-normal frequencies of $E^{N}_{-,-,+,+}$, which are further dominated by the tangential-normal frequencies of $E^{N}_{-,-,+,+}$. 

To apply Cor \ref{cor: isochron nondeg} , it is thus sufficient to verify the non-degeneracy condition for the tangential-normal frequencies $\nu_{per}, \nu_{n, 1}, \nu_{n, 2}, \nu_{n, 3}, \nu_{n, 4}$ of $E_{-,-,+,+}$. Moreover, according to Cor \ref{cor: proper deg isochron nondeg}, it is enough to verify the non-degeneracy conditions separately for frequencies in different scales. To avoid serious computational difficulties, we choose as parameters the rescaled masses $\bar{m}_{1}, \bar{m}_{2}, \bar{m}_{3}$, and the eccentricity $e_{2}$: 

1. Frequency of the periodic orbit: the required non-degeneracy condition is just $\dfrac{\partial F_{Kep}}{\partial D_{3}} \neq 0$.

2. Normal frequencies: This breaks down to the non-trivial dependence of the nontrivial coefficients of the monic quadratic equations with respect to parameters (the eccentricity $e_{2}$ and the masses $\bar{m}_{1}, \bar{m}_{2}, \bar{m}_{3}$). We write Eqs \ref{eq: 1st quadratic} and \ref{eq: 2nd quadratic} respectively as
\begin{equation}
x^{2}+b_{1} x +c_{1}=0
\end{equation}
and
\begin{equation}
x^{2}+b_{2} x +c_{2}=0
\end{equation}
respectively.

\begin{lem}\label{Lem: Normal Nondegeneracy}The Jacobians of $(b_{1}, c_{1})$ with respect to $(\bar{m}_{1}, e_{2})$ and of $(b_{2}, c_{2})$ with respect to $(\bar{m}_{2}, \bar{m}_{3})$ are both non-degenerate almost everywhere.
\end{lem}
\begin{proof} Assisted by Maple 16, we find that 
\begin{itemize}
\item$\left|\dfrac{\partial (b_{1}, c_{1})}{\partial (m_{1}, e_{2})}\right|$ evaluated at $(m_{1}=1, m_{2}=m_{3}=0)$ equals to $8 \bar{B}^{3} \bar{A}$;
\item $\left|\dfrac{\partial (b_{2}, c_{2})}{\partial (m_{2}, m_{3})}\right|$ evaluated at $(m_{1}=1, m_{2}=m_{3}=0)$ equals to $-16 \cdot 2^{5/6} \bar{B}^{10}$. 
\end{itemize}
The conclusion thus follows from analyticity and the fact that $m_{1}= \mu \bar{m}_{1}, m_{2}= \mu \bar{m}_{2}, m_{3}= \mu \bar{m}_{3}$.
\end{proof}

Therefore, for almost every fixed rescaled masses $\bar{m}_{1}, \bar{m}_{2}, \bar{m}_{3}$ and fixed $e_{2}$, there exists $\bar{\delta}_{0}>0$, such that in a $\bar{\delta}_{0}$-neighborhood of the periodic orbits with fixed parameter $e_{2}$ in the family $D^{N}_{-,-,+,+}$, every Lagrangian tori of $F_{Kep}+F_{res,l}^{N}$ has full torsion, in the sense that its frequency map (which is a small perturbation of the frequency map of the corresponding periodic orbit in the family $D_{-,-,+,+}$) is a local diffeomorphism. 

\begin{theo}\label{Theo: KAM tori}
Having fixed $m_{0}$, there exists $\mu_{0}>0$, such that for almost all masses $m_{1}, m_{2}, m_{3}$ satisfying 
$\max \{m_{1}, m_{2}, m_{3}\} \le \mu_{0} m_{0},$
 there exists a set $\Lambda_{m_{1}, m_{2}, m_{3}}$ of positive measure consisting of invariant Lagrangian tori of $F_{Kep}+F_{res,l}^{N}$ close to the continued family $D^{N}_{-,-,+,+}$, such that for any invariant Lagrangian torus in $\Lambda_{m_{1}, m_{2}, m_{3}}$, there exists a set of positive measure of masses $\mu \bar{m}'_{1}, \mu \bar{m}'_{2}, \mu \bar{m}'_{3}$ with $\bar{m}'_{1}, \bar{m}'_{2}, \bar{m}'_{3}$ close respectively to $\bar{m}_{1}, \bar{m}_{2}, \bar{m}_{3}$, such that this invariant torus, with small deformation, persists under perturbation and give rise to an invariant torus of the system $F$ with mass parameter $m'_{1}, m'_{2}, m'_{3}$.  These invariant tori of $F$ form a set of positive measure in the direct product of the phase space with the space of masses $(m_{1}, m_{2}, m_{3})$.
\end{theo}
\begin{proof} Let $(I_{1}, \theta_{1}, I_{2}, \theta_{2}, I_{3}, \theta_{3}, I_{4}, \theta_{4})$ be a set of action-angle coordinates of the fully reduced system of $F_{Kep}+F_{res,l}^{N}$ in the deleted neighborhood of $E^{N}_{-,-,+,+}$ such that for $i=1,2,3,4$, the frequencies $\nu^{N}_{n, i}$ of $\theta_{i}$ is close to the frequencies $\nu_{n, i}$ respectively in the sense that $\nu^{N}_{n,i}=\nu_{n,i}+O(\mu)$. We localize ourselves in an $\bar{\delta}$-neighborhood of $E^{N}_{-,-,+,+}$, so that the action variables $I_{1}, I_{2}, I_{3}, I_{4} \sim \bar{\delta}$. In the system $F_{Kep}+F_{res,l}^{N}$ only reduced by the rotational symmetry, such a neighborhood corresponds to a neighborhood of the family $D^{N}_{-,-,+,+}$ of periodic orbits in which the invariant tori are obtained by fixing $D_{3}$ and $I_{1}, I_{2}, I_{3}, I_{4}$. 

We see from Lem \ref{Lem: Normal Nondegeneracy} that for almost all $\bar{m}_{1}, \bar{m}_{2}, \bar{m_{3}}$, there exists an open set $\Lambda$ in the phase space on which the map 
$$(D_{3}, \bar{m}_{1}, \bar{m}_{2}, \bar{m}_{3}, e_{2}) \mapsto (\nu_{per}, \nu_{n,1}, \nu_{n,2}, \nu_{n,3}, \nu_{n,4})$$ is non-degenerate. Let $\mu, e$ and $\bar{\delta}$ be small enough such that the frequency map 
$$(D_{3}, \bar{m}_{1}, \bar{m}_{2}, \bar{m}_{3}, e_{2}) \mapsto (\nu_{per}, \nu^{N}_{n,1}, \nu^{N}_{n,2}, \nu^{N}_{n,3}, \nu^{N}_{n,4})$$
is non-degenerate on an open set $\Lambda^{o}$. Let $\mathcal{T}$ be any invariant torus in $\Lambda^{o}$.  It is now enough to apply Thm \ref{KAM}  together with the help of Cor \ref{cor: proper deg isochron nondeg} to $\mathcal{T}$ for enough small parameters $0< \mu <<e<<1$, small enough $\bar{\delta}>0 $, and large enough $N$. 

We now treat both the hypothetical frequency mapping $\omega_{\iota}$ and the unperturbed frequency mapping $\omega^{o}_{\iota}$ as functions defined on the corresponding part of the direct product of the phase space with the space of masses $(m_{1}, m_{2}, m_{3})$. Since $\omega_{\iota}$ is $C^{\infty}$ close to the unperturbed frequency mapping $\omega^{o}_{\iota}$, their regular fibers have the same dimension. The last assertion thus follows from Cor \ref{cor: isochron nondeg} and the Fubini theorem.  
\end{proof}

 The following corollary follows from the above theorem and the Fubini theorem.

\begin{cor} There exists $\mu_{0}>0$, such that for almost all masses $m_{1}, m_{2}, m_{3}$ satisfying $\max \{m_{1}, m_{2}, m_{3}\} \le \mu_{0} m_{0}$, there exists a set of positive measure of invariant Lagrangian tori of $F$ in a small neighborhood of the continued family of linearly stable periodic orbits from the family $D_{-,-,+,+}$.
\end{cor}

\begin{rem} To verify the iso-chronic non-degeneracy condition near such a periodic orbit for a given set of masses, we shall have to calculate the corresponding Birkhoff invariants. Already with this relatively simple model, a direct calculation is hard to achieve even with the help of a computer. 
\end{rem}

\section{Extensions to the Planetary Five-body Problem}\label{Section: A complete system of Galilean satellites}

We now add to the above-mentioned models a fifth far-away satellite ``Callisto'' with a small mass $m_{4}$ and consider the corresponding $1+4$- and $5$-body problems. We assume $m_{4} \sim m_{1}, m_{2}, m_{3} \sim \mu$ and write $m_{4}= \mu \bar{m}_{4}$.

In the rescaled canonical Joviancentric coordinates $(\bar{p}_{1}, \bar{p}_{2},\bar{p}_{3}, \bar{p}_{4}, \tilde{q}_{1}, \tilde{q}_{2}, \tilde{q}_{3}, \tilde{q}_{4})$ (in which $\tilde{q}_{4}$ designates the relative position of the fourth satellite with respect to Jupiter, $\tilde{p}_{4}$ is its linear momentum and $\tilde{p}_{4}=\mu \bar{p}_{4}$), the Hamiltonian $\widetilde{F}$ of either of the systems (reduced from the translation symmetry in the five-body case) is analogously decomposed as 
$$\widetilde{F}=\widetilde{F}_{Kep}+\widetilde{F}_{pert},$$
with
\begin{itemize}
\item $\widetilde{F}_{Kep}=F_{Kep} + F_{Kep, 4}$,
\item $\widetilde{F}_{pert}=F_{pert}+ F_{pert, 4}$.
\end{itemize}
In which 
$$F_{Kep, 4}=\dfrac{\|\bar{p}_{4}\|^{2}}{\mu_{4}}-\dfrac{\mu_{4} M_{4}}{r_{04}},\quad F_{pert, 4}=\mu \sum_{i=1}^{3} \dfrac{\bar{p}_{i} \bar{p}_{4}}{m_{0}}-\mu \sum_{i=1}^{3} \dfrac{\bar{m}_{i} \bar{m}_{4}}{r_{i4}}.$$


To allow further analysis, we set 
\begin{equation*}
\left\{
\begin{array}{ll}
&D_{1}=L_{1}, \, \, \, \, \, \, \,  \quad \quad \quad \qquad \qquad \delta_{1}=l_{1}-2 l_{2} , \\
&D_{2}=2 L_{1}+L_{2}, \,\,\,\, \quad \qquad \qquad \delta_{2}=l_{2}-2 l_{3}, \\
&D_{3}=4 L_{1}+2 L_{2}+L_{3},\, \quad \qquad \delta_{3}=l_{3},\\
&\Lambda_{4}=L_{4} \,\quad \quad \qquad \qquad \qquad \, \, \, \lambda_{4}=l_{4}+g_{4}-g_{3},\\
&Z_{1}=G_{1},  \,\,\, \quad \quad \quad \quad \qquad \qquad \eta_{1}=g_{1}-g_{2} , \\
&Z_{2}=G_{1}+G_{2}, \,\,\quad \quad \qquad \qquad \eta_{2}=g_{2}-g_{3}, \\
&Z'_{3}=G_{1}+G_{2}+G_{3}+G_{4},  \qquad \eta_{3}=g_{3}.\\
&Z_{4}=L_{4}-G_{4}, \, \, \, \, \quad \qquad \quad \qquad \eta_{4}=-(g_{4}-g_{3}).\\
\end{array}\right.
\end{equation*}

and, following Poincaré, we further take $$\xi_{4} + i \eta_{4}= \sqrt{2(L_{4}-G_{4})} e^{-i(g_{4}-g_{3})}$$
to have a set of regular symplectic coordinates in the neighborhood circular motions of the fourth body.

We have thus obtained a set of Darboux coordinates 
$$(D_{1}, \delta_{1}, D_{2}, \delta_{2}, D_{3}, \delta_{3}, \Lambda_{4}, \lambda_{4}, Z_{1}, \eta_{1}, Z_{2}, \eta_{2}, Z'_{3}, \eta_{3}, \xi_{4}, \eta_{4}),$$
which is regular up to circular orbit of the fourth body as long as the others are not.

Recall that $\nu_{Kep, 3}$ is the Keplerian frequency of the third satellite. We denote by $\nu_{Kep, 4}$ the Keplerian frequency of the fourth satellite.

With the hypothesis made in Subsection \ref{Subsection: Elimination procedure}, we assume in addition that 
$a_{3}<a_{4}, 0\le e_{4}< e^{\wedge}<1$ such that 
$$\dfrac{a_{3} (1+e^{\wedge})}{a_{4}(1-e^{\wedge})}<1$$
so that all the four elliptic orbits are bounded away from each other. The set $\mathcal{Q}$ defined by these conditions is identified, by the above-mentioned Darboux coordinates, to a subset $\tilde{\mathcal{Q}}$ of $\T^{7} \times \R^{7} \times \R^{2}$.

Suppose that the set
$$\{D_{1}=D_{1}^{0}, D_{2}=D_{2}^{0}, D_{3}=D_{3}^{0}\}$$
consists in Keplerian motions in $\tilde{\mathcal{Q}}$ for which the inner three frequencies satisfies the $4:2:1$-resonance. For $C_{1}>0,$ let $\tilde{\mathcal{M}}$ to be the transversal-Cantor subset of the $C_{1} \mu$-neighborhood of this set in $\tilde{\mathcal{Q}}$ in which $(\nu_{Kep, 3}, \nu_{Kep,4})$ is $(\bar{\gamma}, \bar{\tau})$-Diophantine for some $0<\bar{\gamma}<<1, \bar{\tau}>1$. The fact that $\tilde{\mathcal{M}}$ has positive measure for $\bar{\tau}>1$ and small enough $\bar{\gamma}$ will be a consequence of the non-degeneracy condition of the Keplerian part with respect to $D_{3}$ and $\Lambda_{4}$. This is part of the non-degeneracy conditions for application of KAM therem in the sequel, and is direct to verify by using the explicit formula of $\widetilde{F}_{Kep}$. 

We shall assume that $C_{1}$ is chosen large enough to allow further application of KAM theorems. We let $\check{\mathcal{Q}}$ to be the $C_{2}$-neighborhood of $\{D_{1}=D_{1}^{0}, D_{2}=D_{2}^{0}, D_{3}=D_{3}^{0}, \Lambda_{4}=\Lambda_{4}^{0}\}$ in $\tilde{\mathcal{Q}}$ for some small enough $C_{2}$. The small parameter $\mu$ is supposed to satisfy $C_{1} \mu < C_{2}$.

Being analytic functions on $\check{Q}$, there exist $s>0$, such that $F_{Kep}$ and $F_{pert}$ extend to analytic functions on $T_{\check{Q}, s}$.

For a function $f: \T^{7} \times \R^{7} \times \R^{2} \to \R$, we define
$$\langle f \rangle_{\delta_{3}, \lambda_{4}}=\dfrac{1}{4 \pi^{2}} \int_{0}^{2 \pi} f d \delta_{3} d \lambda_{4}. $$

\begin{prop}\label{prop: elimination 2} For any integer $N$, there exists an $O(\mu)$-$C^{\infty}$-Whitney symplectic transformation $\psi: \tilde{\mathcal{M}} \to \phi(\tilde{\mathcal{M}})$ such that
$$\psi^{*} \tilde{F} = \tilde{F}_{Kep}+\tilde{F}_{res}+\tilde{F}_{rem},$$
in which in $\tilde{\mathcal{N}}$ the analytic functions
\begin{itemize}
\item$ \tilde{F}_{res}=\langle \tilde{F}_{res} \rangle_{\delta_{3}, \lambda_{4}}$.
\item $\tilde{F}_{rem}=O(\mu e^{N}+\mu^{N})$.
\end{itemize}
\end{prop}
\begin{proof} The essential part of the proof is to consecutively eliminate the angles $\delta_{3}, \lambda_{4}$ with an elimination procedure analogous to the proof of Prop \ref{prop: elimination}. In $ \tilde{\mathcal{M}}$, under the Diophantine condition on the third and fourth frequencies, the existence of an $C^{\infty}$-Whitney Hamiltonian solving the modified cohomological equation (defined again by neglecting the $O(\mu)$-terms involving $\nu_{Kep,1}$ and $\nu_{Kep,2}$) follows from \emph{e.g.} \cite[Prop 4]{Fejoz2011KAM}. Finally, it is enough to truncate the resulting normal form at the $(N-1)$-th order of the eccentricities.
\end{proof}

\begin{lem}(Herman, see \cite[Lemme 64]{FejozStability}) The indirect part $\mu \sum_{i=1}^{3} \dfrac{\bar{p}_{i} \bar{p}_{4}}{m_{0}}$ does not contribute to the secular system in the sense that for $i=1,2,3$
$$\int_{\T^{2}} \mu \dfrac{\bar{p}_{i} \bar{p}_{4}}{m_{0}} d l_{i} d l_{4}=0.$$
\end{lem}

The lemma is proven by observing that for $i=1,2,3,4$, $\bar{p}_{i}=\bar{m}_{i} \dot{\tilde{q}}_{i}$ while $\tilde{q}_{i}$ is a periodic function of $l_{i}$ along Keplerian elliptic orbits, and application of the Fubini theorem.

The Motion of Callisto is thus dominated by the corresponding secular system $F_{sec, 4}(\xi_{4}, \eta_{4})$, obtained from averaging the function 
$$F_{pert,4}=-\dfrac{\mu \bar{m}_{1} \bar{m}_{4}}{r_{14}}-\dfrac{\mu \bar{m}_{2} \bar{m}_{4}}{r_{24}}-\dfrac{\mu \bar{m}_{3} \bar{m}_{4}}{r_{34}}$$
over the fast angles $\delta_{3}$ and $\lambda_{4}$. We set 
$$\bar{F}_{sec, 4} = \int_{\T^{4}} F_{pert,4} d \delta_{3} d \lambda_{4}.$$
This is a function of order $O(\mu e^{2})$ (\cite[p. 405]{Tisserand1889}). In particular, it is of higher order compared to, and consequently dominated by $F_{res}$. We evaluate $\bar{F}_{sec, 4}$ at the corresponding circular orbits of the inner three to obtain a function $F_{sec, 4}$.

\begin{claim} The point $(0, 0)$ is an elliptic equilibrium of $F_{sec, 4}(\xi_{4}, \eta_{4})$.
\end{claim}

This follows from the fact that the quadratic part of the analytic function $F_{sec, 4}(\xi_{4}, \eta_{4})$ is even in $(\xi_{4}, \eta_{4})$. More precisely, we deduce from \cite[p. 405]{Tisserand1889} that this quadratic part reads
$$\mu \bar{m}_{4} \Bigl(\bar{m}_{1}  \frac{1}{2} B^{(1)} (a_{1}, a_{4}) + \bar{m}_{2}  \frac{1}{4} B^{(1)} (a_{2}, a_{4})+\bar{m}_{3}  \frac{1}{8} B^{(1)} (a_{3}, a_{4})\Bigr)(\xi_{4}^{2}+\eta_{4}^{2}).$$
in which (\cite[p. 270, 271]{Tisserand1889}) $B^{(1)} (a, a')=\frac{a}{a'^{2}} b^{1}_{3/2}(\frac{a}{a'})$.

We thus arrive at an approximating system:
$$\tilde{F}_{Kep}+F_{res}+F_{sec,4}.$$
We may thus identify the periodic solutions of the three satellites together with an circular orbit of the fourth satellite, which is used by De Sitter in \cite{DeSitter1925} as ``intermediate orbits''. After symplectically reduced by the $SO(2) \times SO(2)$-symmetry of shifting $l_{3}, \lambda_{4}$, these solutions descends to an non-degenerate equilibrium of the reduced system, which can be continued to
$$\tilde{F}_{Kep}+\tilde{F}_{res}$$
for any $N$ and small $e, \mu$, with normal frequencies dominated by the normal frequencies of the corresponding equilibrium in the symplectically reduced system of $\tilde{F}_{Kep}+F_{res}+F_{sec,4}$ by this $SO(2) \times SO(2)$-symmetry. Without reducing this $SO(2) \times SO(2)$-symmetry, this equilibrium corresponds to normally elliptic invariant 2-tori of the system $\tilde{F}_{Kep}+F_{res}+F_{sec,4}$ with an neighborhood  by nearby librating invariant Lagrangian 7-tori.

We consider the persistence of these invariant tori under small perturbation $O(\mu e^{N}+\mu^{N})$ for large enough $N$. We shall take the same parameters for the inner three satellites as in Section \ref{Sec: KAM near periodic orbit}. The function $F_{sec, 4}(\xi_{4}, \eta_{4})$ contains the normalized mass $\bar{m}_{4}$ as factor. To apply KAM theorem, we take $\bar{m}_{4}$ and $\Lambda_{4}$ as additional parameters. To see that the non-degeneracy condition is satisfied, in view of Cor \ref{cor: proper deg isochron nondeg} it is enough to notice further that this elliptic equilibrium is non-degenerate with normal frequency of order $\mu$, and $F_{Kep,4}$ is non-degenerate with respect to $\Lambda_{4}$.

We have thus proved the following theorems, by imposing $0<\mu<<e<<1$ and  application of Thm \ref{KAM} together with the help of Cor \ref{cor: proper deg isochron nondeg} in the same spirit as in the previous section:

\begin{theo}\label{Theo: isotropic KAM tori with Callisto} When $m_{0}$ is fixed, or almost all enough small masses $m_{1}, m_{2}, m_{3}, m_{4}$ and eccentricities $e_{2}, e_{4}$, the corresponding normally elliptic invariant 2-torus persists to exist for the system $\tilde{F}$ with mass parameter $m'_{1}, m'_{2}, m'_{3}, m_{4}$ close to $m_{1}, m_{2}, m_{3}, m_{4}$, provided the small parameters $\mu, e, \bar{\delta}$ are small enough. 
\end{theo}

\begin{theo}\label{Theo: KAM tori with Callisto} When $m_{0}$ is fixed, for any fixed enough small masses  $m_{1}, m_{2}, m_{3}, m_{4}$, there exists a set $\Lambda_{m_{1}, m_{2}, m_{3}, m_{4}}$ of positive measure consisting of invariant Lagrangian tori of $\widetilde{F}_{Kep}+\widetilde{F}_{res,l}^{c, N}$, such that for any invariant Lagrangian torus in $\Lambda_{m_{1}, m_{2}, m_{3}, m_{4}}$, there exists a set of positive measure of masses $m'_{1}, m'_{2}, m'_{3}, m'_{4}$ close respectively to $m_{1}, m_{2}, m_{3}, m_{4}$ such that this invariant torus, with small deformation, persists to exist for the system $F^{c}$ with mass parameter $m'_{1}, m'_{2}, m'_{3}, m'_{4}$, provided the small parameters $\mu, e, \bar{\delta}$ are small enough. These invariant KAM tori of $\widetilde{F}$ form a set of positive measure in the direct product of the phase space with the space of masses $(m_{1}, m_{2}, m_{3}, m_{4})$. 
\end{theo}

We deduce the following corollary from the above theorem by Fubini theorem.

\begin{cor} For almost all masses $(m_{1}, m_{2}, m_{3}, m_{4})$, there exists a set of positive measure of invariant Lagrangian tori of $\widetilde{F}$ in a small neighborhood of the continued normally elliptic invariant 2-tori.
\end{cor}

\bibliography{DeSitter}
\smallskip

\noindent Henk Broer, Rijksuniversiteit Groningen: h.w.broer@rug.nl \\
Lei Zhao, \quad Rijksuniversiteit Groningen:  l.zhao@rug.nl \\
\end{document}